\documentclass[onefignum,onetabnum,oneeqnum,onethmnum]{siamart190516}

\usepackage{amsmath}
\usepackage{fullpage}
\usepackage{lipsum}
\usepackage{amsfonts}
\usepackage{amssymb}
\usepackage{graphicx}
\usepackage{epstopdf}
\usepackage{algorithmic}
\usepackage{paralist}
\usepackage{mathtools}
\ifpdf
  \DeclareGraphicsExtensions{.eps,.pdf,.png,.jpg}
\else
  \DeclareGraphicsExtensions{.eps}
\fi

\newsiamremark{remark}{Remark}
\newsiamremark{hypothesis}{Hypothesis}
\newsiamthm{claim}{Claim}
\usepackage{comment}
\graphicspath{{fig/}{../fig/}} 
\numberwithin{theorem}{section}

\numberwithin{dfn}{section}

\newtheorem{rmk}{Remark}
\numberwithin{rmk}{section}

\newcommand{\ba}{\begin{array}}
\newcommand{\ea}{\end{array}}
\newcommand{\bea}{\begin{eqnarray}}
\newcommand{\eea}{\end{eqnarray}}
\newcommand{\be}{\begin{equation}}
\newcommand{\ee}{\end{equation}}
\newcommand{\bd}{\begin{displaymath}}
\newcommand{\ed}{\end{displaymath}}
\newcommand{\bi}{\begin{itemize}}
\newcommand{\ei}{\end{itemize}}
\newcommand{\bn}{\begin{enumerate}}
\newcommand{\en}{\end{enumerate}}

\newcommand{\vv}[1]{\boldsymbol{#1}}

\begin{document}

% Following siamart/ex_article.tex, this seems needed:
%\newcommand{\TheTitle}{
%A new high-order method for evaluation of Laplace double layer potentials close to their sources in 3D}
%Close evaluation of Laplace layer potentials via Stokes theorem on manifolds}  
%Computing Laplace layer potentials via Stokes theorem on manifolds}  
%Close evaluation of Laplace layer potentials in three dimensions}
%\newcommand{\TheAuthors}{H. Zhu, and S. Veerapaneni}  % heading
%\headers{Close evaluation for Laplace layer potentials on manifolds}{\TheAuthors}

%% ------------------------------------------------------------------
%% HEADING INFORMATION
%% ------------------------------------------------------------------
\title{ %A differential geometric approach to high-order close evaluation of Laplace layer potentials
High-order close evaluation of Laplace layer potentials: A differential geometric approach 
%Close evaluation of Laplace layer potentials: A differential geometric approach to construct high-order numerical schemes
%Close Evaluation of Laplace layer potentials via Stokes theorem on manifolds
%\thanks{Submitted to the editors DATE.
%\funding{This work was supported by NSF under grants DMS-1719834 and DMS-1454010 and the Mcubed program at the University of Michigan. The work of SV was also supported by the Flatiron Institute, a division of the Simons Foundation.}}
}

%{\TheTitle}\thanks{
%\funding{This work was supported by NSF under grants DMS-1719834 and DMS-1454010 and the Mcubed program at the University of Michigan. The work of SV was also supported by the Flatiron Institute, a division of the Simons Foundation.}}}
\author{Hai Zhu%
  \thanks{Department of Mathematics,
  University of Michigan, Ann Arbor, MI 48109 (\email{hszhu@umich.edu}, \email{shravan@umich.edu}).}%
  \and
  Shravan Veerapaneni\footnotemark[1]%
}

%% ------------------------------------------------------------------
%% END HEADING INFORMATION
%% ------------------------------------------------------------------

%% ------------------------------------------------------------------
%% MAIN Document
%% ------------------------------------------------------------------
\maketitle
%% ------------------------------------------------------------------
%% ABSTRACT
%% ------------------------------------------------------------------
\begin{abstract}
This paper presents a new approach for solving the {\em close evaluation problem} in three dimensions, commonly encountered while solving linear elliptic partial differential equations via potential theory. The goal is to evaluate layer potentials close to the boundary over which they are defined. The approach introduced here converts these nearly-singular integrals on a patch of the boundary to a set of non-singular line integrals on the patch boundary using the Stokes theorem on manifolds. A function approximation scheme based on harmonic polynomials is designed to express the integrand in a form that is suitable for applying the Stokes theorem. As long as the data---the boundary and the density function---is given in a high-order format, the double-layer potential and its derivatives can be evaluated with high-order accuracy using this scheme both on and off the boundary. In particular, we present numerical results demonstrating seventh-order convergence on a smooth, warped torus example achieving 10-digit accuracy in evaluating double layer potential at targets that are arbitrarily close to the boundary. 

%n addition, we demonstrate the versatility of the method by on a triangulated bunny surface taken from a standard mesh library.
 
%We develop a new high-order singular kernel evaluation scheme in three-dimensions. The method relies on the use of exterior calculus and interpolations of density function and manifolds in terms of harmonic polynomials. This paper focuses on the numerical evaluation of singular and nearly singular Laplace boundary integral operators. A variety of numerical examples demonstrate the effectiveness of the new scheme at points that lie arbitrarily close to the surface. The spatial discretization uses Gaussian quadrature on panels. Depending on the order of density and panel surface interpolation, our new scheme serves as a high order local correction to any straight-up quadrature. In addition, we show that our method is compatable with standard surface triangulation.
 
\end{abstract}

\begin{keywords}
potential theory, singular integrals, harmonic polynomials, exterior calculus, high-order methods
\end{keywords}

\section{Introduction}
\label{s:intro}
In this paper, we describe a high-order accurate numerical algorithm for evaluating the double-layer potential (DLP) for Laplace equation given by
	\begin{equation} \label{eq:introLapDLP}
		\begin{aligned}
			\mathcal{D}[\mu](\boldsymbol{r}') = \int_{\mathcal{M}} \frac{\partial G(\boldsymbol{r}'-\boldsymbol{r})}{\partial \boldsymbol{n}_{\boldsymbol{r}}}\mu(\boldsymbol{r})dS_{\boldsymbol{r}} 
		\end{aligned}
	\end{equation}	
where $G(\boldsymbol{r}',\boldsymbol{r})=1/4\pi|\boldsymbol{r}'-\boldsymbol{r}|$ is the Green's function for the Laplace equation, $\mu(\boldsymbol{r})$ is a density function, $\mathcal{M}$ is a closed two-dimensional manifold in $\mathbb{R}^3$ and $\boldsymbol{n}_{\boldsymbol{r}}$ is its normal. Layer potentials such as the DLP satisfy the underlying partial differential equation (PDE) by construction and are often employed in mathematical analysis and numerical solution of PDEs \cite{kress1989linear}. Fast and accurate numerical schemes for (\ref{eq:introLapDLP}) are fundamentally important owing to the ubiquity of Laplace equation in sciences and engineering. Moreover, they serve as templates for other linear elliptic PDE solvers via potential theory. 

In practical applications, one needs to evaluate \eqref{eq:introLapDLP} both on and off the surface $\mathcal{M}$. If the target $\boldsymbol{r}'$ is located far off the surface, a smooth quadrature rule designed for the given surface representation can be applied efficiently. However, the integrand in \eqref{eq:introLapDLP} becomes weakly-singular for on-surface targets and nearly-singular for targets located close to the surface. In both cases, specialized quadrature rules are necessary to achieve desired order of accuracy. While the subject of developing high-order rules for weakly-singular integrals is a classical one, nearly-singular integration is an active area of research. For two-dimensional problems (where $\mathcal{M}$ is a curve on the plane), significant progress has been made on accurate evaluation schemes for 
nearly-singular integrals, some recent works include  \cite{rachh2017fast, carvalho2018asymptotic, rahimian2018ubiquitous, af2018adaptive,  wu2019solution} (also see references therein). 
In contrast, fewer number of works exist for high-order close evaluation in the case of three-dimensional problems, owing to the complexity of handling a stronger kernel singularity over high-order surface meshes. 
 \begin{figure}[!ht]
    \centering
    \includegraphics[height=0.28\textwidth]{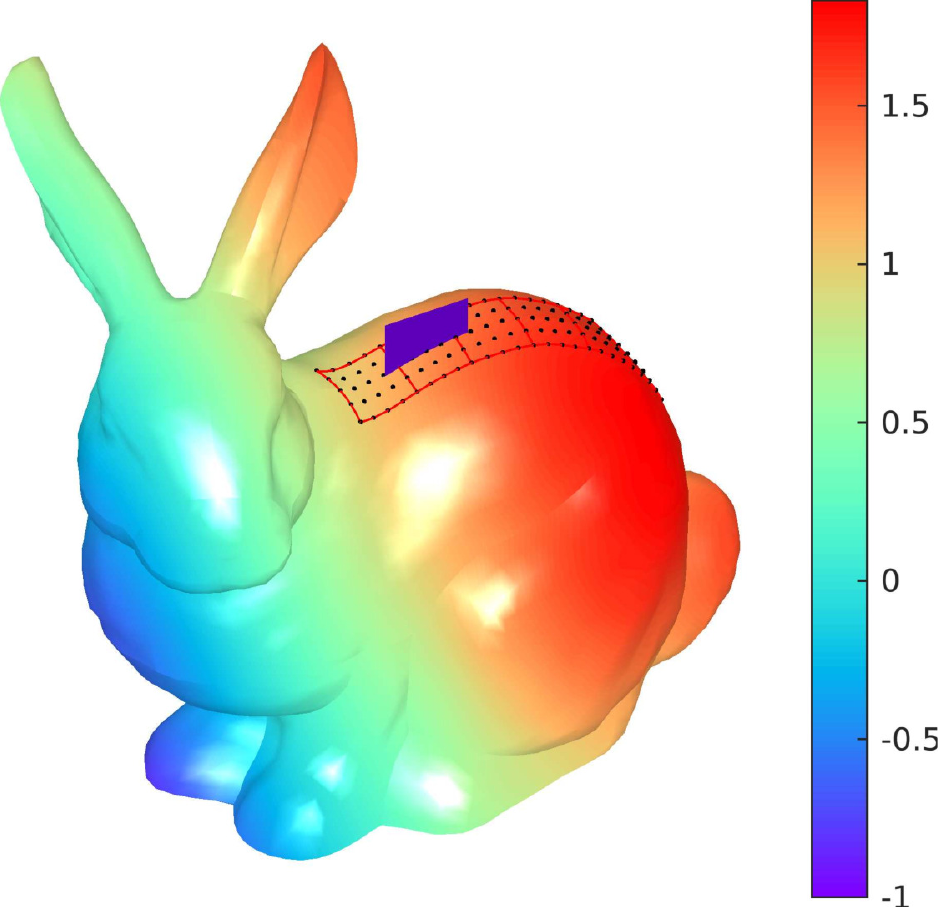}
    \includegraphics[height=0.28\textwidth]{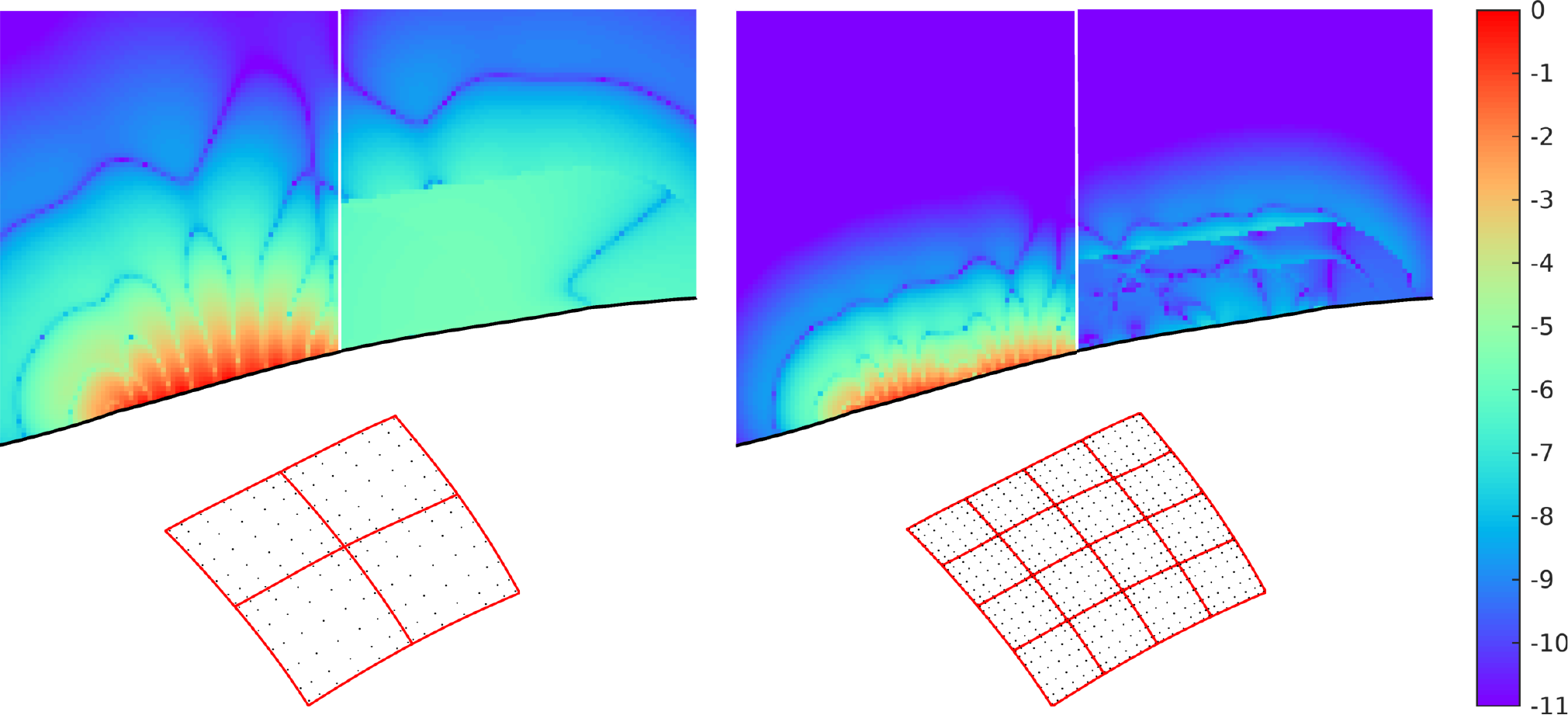}
    \caption{ One of the key advantages of the close evaluation scheme developed in this paper is its ease of handling arbitrary meshes. Here, we demonstrate its performance on the Stanford bunny triangulation data \cite{turk1994zippered}. We used the interactive sketch-based quadrangulation method of \cite{takayama2013sketch} to create high-quality quad remeshings locally as shown on the top of the bunny. We evaluate the DLP at targets that are located arbitrarily close to the surface as shown on top in blue color. The surface is colored by the density function $\mu$, which  was set as  $\mu(x,y,z)=e^{xy}-1+x+\sin(x^4+1/2y^3)+y-1/2y^2+1/5y^6+z$. (Middle and right) Given this setup, we demonstrate the performance of the new scheme by considering one of the quads, successively refining it two-fold and visualizing the errors due to direct evaluation of DLP via high-order smooth quadrature rule (left half) and the new close evaluation scheme (right half). We note that while the errors stagnate in a band close to the surface in the case of smooth quadrature, the new scheme achieves uniform accuracy upto 10-digits. More details on this experiment are provided in Section \ref{NumExp}.} \label{fig:bunny}
\end{figure}

{\em Synopsis of the new approach.} Consider a subdomain $D \subset \mathcal{M}$. A surface integral on $D$ can be converted into a line integral on $\partial D$ using the Stokes theorem on manifolds as long as the integrand is an exact form \cite{spivak2018calculus}. Clearly, this condition is not necessarily satisfied in the case of DLP (\ref{eq:introLapDLP}) for an arbitrary $\mu$. The main idea here is that we can construct basis functions for approximating $\mu$ in $D$, which when multiplied by the kernel in (\ref{eq:introLapDLP}) are exact forms. Thereby, when the target $\boldsymbol{r}'$ is close to $D$, we can apply this procedure to convert a nearly-singular surface integral on $D$ to a non-singular line integral on $\partial D$ (assuming $\boldsymbol{r}'$ is not close to $\partial D$). In this paper, we construct such basis functions using harmonic polynomials and quaternion algebra. The scheme is relatively insensitive to the underlying high-order surface discretization. Once the density function is expressed in our basis on $D$ (e.g., {\em via} collocation), the layer potential evaluation is carried out in a similar fashion as a product integration scheme, with the caveat that some smooth line integrals on $\partial D$ need to computed numerically in addition. 

{\em Related work.} Here, we restrict our discussion to closely related recent works; a more extensive literature survey on singular and near-singular integration schemes can be found in \cite{hao2014high, perez2019harmonic, morse2020robust}. In the first class of methods, the issue of close evaluation is overcome by exploiting the smoothness of DLP away from $\mathcal{M}$. In the quadrature-by-expansion (\texttt{QBX}) scheme, originally proposed in \cite{barnett2014evaluation,klockner2013quadrature}, the DLP is approximated at centers away from $\mathcal{M}$ using high-order local expansions, which are valid at points closer to or on $\mathcal{M}$. Extension of \texttt{QBX} to three-dimensional problems was recently explored in \cite{siegel2018local, wala2019fast, wala2020optimization}. A related algorithm is the \texttt{hedgehog} scheme of \cite{morse2020robust}, which in turn is an extension of the earlier work by Ying et al.~\cite{ying2006high}. Similar to \texttt{QBX}, \texttt{hedgehog} exploits the smoothness of \eqref{eq:introLapDLP} away from the boundary and evaluates it at carefully chosen ``check'' points along a line passing through the target located close to $\mathcal{M}$ and extrapolates the solution to the target.  

Another popular class of methods are those based on singularity subtraction, wherein, the kernel in \eqref{eq:introLapDLP} is split into a singular part and a smooth part, with the action of the former treated analytically. While low-order variants are often used in practice for three-dimensional problems, high-order extension was presented in \cite{helsing2013higher} for toroidal geometries. Recently, an alternative strategy, termed as harmonic density interpolation (\texttt{HDI}), is presented in  \cite{perez2019harmonic} which focuses on the density instead of the kernel. It regularizes the kernel singularities by splitting the density into two parts: one whose convolution with the kernel can be treated analytically and the other whose derivatives vanish to prescribed order as the target $\boldsymbol{r}'$ approaches the source $\boldsymbol{r}$. Lastly, regularized kernel methods for 3D close evaluation were also developed recently in \cite{beale2015simple, tlupova2019regularized}; high-order accuracy is achieved by introducing correction terms to control the regularization error.

Our approach shares many of the desirable features of \texttt{QBX} and \texttt{hedgehog} schemes including, prominently, the ease of integration with fast algorithms such as the fast multipole method (FMM) \cite{greengard1987fast} since it doesn't modify the kernel and affects the local part evaluation only. On the other hand, the fact that all the computational variables stay on the manifold $\mathcal{M}$ in our scheme offers further advantages such as avoiding the need for optimizing auxiliary parameters like local expansion centers or check points, which may be challenging in situations such as nearly self-touching geometries. Although both \texttt{HDI} and our scheme employ harmonic polynomials for approximating the density function, their usage is fundamentally different in both schemes. In \cite{perez2019harmonic}, harmonic polynomial approximations are sought which cancel the kernel singularity to high-order as $\boldsymbol{r}' \rightarrow \boldsymbol{r}$; it is unclear if such approximations can be constructed to arbitrarily high-order in three-dimensions (\cite{perez2019harmonic} demonstrates third-order convergence). In contrast, our scheme only requires smoothness of the density. Lastly, a key advantage of our approach is that it works on any user supplied meshes without the need for geometry processing; high-order convergence is guaranteed as long as the boundary and the density function are specified in a high-order format (an example is shown in Figure \ref{fig:bunny}). 

In our view, our work is most closely related to the work of Helsing-Ojala \cite{helsing2008evaluation}, which developed a panel-based close evaluation scheme in two dimensions by approximating the density using monomial basis and evaluating their product with the nearly-singular kernels via recurrences. This approach has been shown to offer rapid and accurate solution of several elliptic problems \cite{helsing2008evaluation, barnett2015spectrally, wu2019solution}. The quaternionic harmonic polynomial approximation scheme of the density, introduced in this work, can be viewed as a 3D analogue of their complex monomial approximation scheme. Similarly, we also employ recurrences to evaluate the product of nearly-singular kernels and polynomial basis functions. We note that the differential geometry framework presented here is applicable both for two- and three-dimensional problems, thereby, is a unifying approach.

{\em Limitations.} In this work, we restrict our attention to the Laplace layer potentials only. While our close evaluation scheme can be extended to other linear elliptic PDE kernels, it is by no means trivial: the density approximation scheme needs to be tailored for each individual kernel. We note, however, that there are alternative approaches to directly apply our scheme to other PDE problems---e.g., Stokes potentials can be expressed as a linear combination of Laplace potentials and their derivatives \cite{tornberg2008fast}. Lastly, the method described in this paper cannot be applied directly to globally parameterized surfaces (e.g., spherical harmonic representations). One remedy is to maintain an auxiliary adaptive surface mesh just for the purposes of close evaluation.  

The remainder of this paper is organized as follows. In Section~\ref{ExtC}, we review some preliminaries on exterior calculus and describe the 3D close evaluation problem using the language of exterior calculus. In Section~\ref{Scheme}, the key ideas of the product integration scheme are outlined, followed by a presentation of our quaternionic approximation. Then in Section~\ref{NumerScheme}, we present the overall implementation of our close evaluation scheme. We demonstrate the performance of our algorithm on a variety of examples in Section~\ref{NumExp}, followed by conclusions and discussion on future directions in Section~\ref{Concl}.

\section{Mathematical preliminaries}
\label{ExtC}
The use of exterior calculus greatly simplifies the presentation of our numerical algorithms even though, strictly speaking, is not required for their development. In this section, we review some basic concepts but refer the reader to \cite{arnold2006finite, spivak2018calculus} (or other standard textbooks) for a more thorough introduction to this subject. 

\subsection{Exterior algebra}
	If $V$ is a vector space over $\mathbb{R}$, we will denote by $Alt^k(V)$ the space of alternating $k$-linear maps $V\times\cdots\times V\rightarrow \mathbb{R}$. We refer to such maps as alternating algebraic $k$-forms. A $k$-linear map $\omega\in Alt^k(V)$ is called alternating if
	\begin{equation}
		\begin{aligned}
			\omega(\boldsymbol{v}_1,\cdots,\boldsymbol{v}_i,\cdots,\boldsymbol{v}_j,\cdots,\boldsymbol{v}_k) = -\omega(\boldsymbol{v}_1,\cdots,\boldsymbol{v}_j,\cdots,\boldsymbol{v}_i,\cdots,\boldsymbol{v}_k).
		\end{aligned}
	\end{equation}

	 Thus, an algebraic $k$-form on $V$ assigns to a $k$-tuple $(\boldsymbol{v}_1,\cdots,\boldsymbol{v}_k)$ of elements of $V$ a real number $\omega(\boldsymbol{v}_1,\cdots,\boldsymbol{v}_k)$, with the mapping linear in each argument, and reversing sign when any two arguments are interchanged.
	 
	 Given $\omega\in Alt^k(V)$ and $\eta\in Alt^l(V)$, simple tensor product of $\omega$ and $\eta$ is usually not an alternating algebraic $(k+l)$-form. We instead employ the {\em exterior product} or {\em wedge product} $\omega\wedge\eta\in Alt^{k+l}$, defined by
	 \begin{equation}
	 	\begin{aligned}
	 		&(\omega\wedge\eta)(\boldsymbol{v}_1,\cdots,\boldsymbol{v}_{k+l})\\
	 		&\hspace{0.5in} = \sum_{\sigma\in Sh_{k,l}}\mathrm{sgn}(\sigma)\omega(\boldsymbol{v}_{\sigma(1)},\cdots,\boldsymbol{v}_{\sigma(k)})\eta(\boldsymbol{v}_{\sigma(k+1)},\cdots,\boldsymbol{v}_{\sigma(k+l)}), \quad \boldsymbol{v}_i\in V,
	 	\end{aligned}
	 \end{equation}
	 where $Sh_{k,l}$ is the subset of $(k,l)$ permutations of the set $\{1,2,\cdots,k+l\}$ such that each element $\sigma\in Sh_{k,l}$ satisfies $\sigma(1)<\sigma(2)<\cdots<\sigma(k)$, and $\sigma(k+1)<\sigma(k+2)<\cdots<\cdots<\sigma(k+l)$. The exterior product is both bilinear and associative.
	 
    Throughout this paper, $V$ will be $\mathbb{R}^3$, and $k$ will be $1$ or $2$. In $\mathbb{R}^3$, the canonical basis $\boldsymbol{e}_1, \boldsymbol{e}_2$ and  $\boldsymbol{e}_3$ gives rise to a natural dual basis of $Alt^1(\mathbb{R}^3)$, the space of covectors. This dual basis will often be denoted by $dx, dy$ and $dz$. These basis elements are linear maps (not to be confused with infinitely small change in the variable). For example, 
    $dx (\boldsymbol{v}) = \boldsymbol{e}_1\cdot \boldsymbol{v}=v_1$. 
    The wedge product is also an operation connecting the various $Alt^k(V)$ spaces. 
    For example, the basis for $Alt^2(\mathbb{R}^3)$ can be written using wedge products of $Alt^1(\mathbb{R}^3)$ basis elements as $dx\wedge dy, dy\wedge dz$ and $dz\wedge dx$. 
	 
\subsection{Exterior calculus on manifolds}
	At each point $\boldsymbol{r}$ of a sufficiently smooth manifold $\mathcal{M}$ of dimension $n$, the tangent space $T_{\boldsymbol{r}}\mathcal{M}$ is a vector space of dimension $n$ (in our case, $n=2$). We could think of this as a local coordinate system. If the selection of a vector $\boldsymbol{v}(\boldsymbol{r})$ at $\boldsymbol{r}$ is made in each $T_{\boldsymbol{r}}\mathcal{M}$, we obtain a vector field. 
	
	Applying the exterior algebra construction to the tangent spaces, we obtain the exterior forms bundle $(\boldsymbol{r},\eta)$ with $\boldsymbol{r}\in \mathcal{M}$, $\eta\in Alt^k(T_{\boldsymbol{r}}\mathcal{M})$. A differential $k$-form is a map $\omega$ which associates to each $\boldsymbol{r}\in \mathcal{M}$ an element $\omega_{\boldsymbol{r}}\in Alt^k(T_{\boldsymbol{r}}\mathcal{M})$. If the map $\boldsymbol{r}\in \mathcal{M} \rightarrow \omega_{\boldsymbol{r}}(\boldsymbol{v}_1(\boldsymbol{r}),\cdots,\boldsymbol{v}_k(\boldsymbol{r}))\in \mathbb{R}$ is smooth whenever the $\boldsymbol{v}_i$'s are smooth vector fields, then we say that $\omega$ is a smooth differential $k$-form. We denote by $\Lambda^k(\mathcal{M})$ the space of all smooth differential $k$-forms on $\mathcal{M}$. As $\boldsymbol{r}$ moves around smoothly on $\mathcal{M}$, $\omega$ provides a smoothly varying algebraic $k$-form at each tangent space $T_{\boldsymbol{r}}\mathcal{M}$. 
	
	The exterior product of differential forms can be defined pointwise from exterior product of algebraic forms
	\begin{equation}
		\begin{aligned}
			(\omega\wedge\eta)_{\boldsymbol{r}} = \omega_{\boldsymbol{r}}\wedge \eta_{\boldsymbol{r}}.
		\end{aligned}
	\end{equation}
	
	If $D$ is an oriented submanifold of $\mathcal{M}$, and $\omega$ is a continuous $k$-form, then the integral $\int_D \omega$ is well-defined. 
	
	The exterior derivative $d$ is a linear operator that maps $\Lambda^k(\mathcal{M})$ into $\Lambda^{k+1}(\mathcal{M})$ for each $k\geq 0$. We give a formula for the case $\mathcal{M}$ is a domain in $\mathbb{R}^n$. For given $\omega\in \Lambda^k(\mathcal{M})$ and vectors $\boldsymbol{v}_1,\cdots,\boldsymbol{v}_k$, we obtain a smooth mapping $\mathcal{M}\rightarrow \mathbb{R}$ given by $\boldsymbol{r}\rightarrow \omega_{\boldsymbol{r}}(\boldsymbol{v}_1,\cdots,\boldsymbol{v}_k)$. We then define
	\begin{equation}
		\begin{aligned}
			d\omega_{\boldsymbol{r}}(\boldsymbol{v}_1,\cdots,\boldsymbol{v}_k) = \sum_{j=1}^{k+1}(-1)^{j+1}\partial_{\boldsymbol{v}_j}\omega_{\boldsymbol{r}}(\boldsymbol{v}_1,\cdots,\boldsymbol{\hat{v}}_j,\cdots,\boldsymbol{v}_{k+1}),
		\end{aligned}
	\end{equation} 
	where the hat is used to indicate a suppressed argument. If $\omega\in \Lambda^k(\mathcal{M})$ and $\eta\in \Lambda^l(\mathcal{M})$, then
	\begin{equation}
		\begin{aligned}
			d(\omega\wedge\eta) = d\omega\wedge\eta + (-1)^k\omega\wedge d\eta.
		\end{aligned}
	\end{equation}

\subsection{Integral equation formulation} 
    Consider the following interior Dirichlet problem for the Laplace equation in a three-dimensional domain $\mathbb{D}$ bounded by $\mathcal{M}$,
    \begin{equation} \label{eq:laplace_bvp}
     \Delta u = 0 \quad \text{in} \quad \mathbb{D}, \quad u = g \quad\text{on}\quad\mathcal{M}.    
    \end{equation}
    We can employ an indirect integral equation formulation \cite{kress1989linear} for solving this problem, wherein, we set $u(\vv{r}') = \mathcal{D}[\mu](\vv{r}')$, the double-layer potential as defined in \eqref{eq:introLapDLP}. This ansatz satisfies the Laplace equation by construction and enforcing the boundary condition yields the following boundary integral equation for the unknown $\mu$:
    \begin{equation} \label{eq:laplace}
    -\frac{1}{2}\mu (\vv{r}') + \mathcal{D}[\mu](\vv{r}') = g(\vv{r}'), \quad \forall \quad \vv{r}' \in \mathcal{M}, 
     \end{equation}
    where the evaluation of $\mathcal{D}$ on $\mathcal{M}$ is performed in the principal value sense. Solving this BIE for $\mu$, one can evaluate the solution $u$ at any target in the domain by using (\ref{eq:introLapDLP}). Similarly, other Laplace boundary value problems can be recast as BIEs using potential theory (e.g., see \cite{kress1989linear}). 
    
    Now, let's express the DLP evaluation as integration of differential forms.  On the manifold $\mathcal{M}$, we have \cite[Thm.~5-6]{spivak2018calculus}:
    	\begin{equation}\label{eq:NormalArea}
    			n_1 dS_{\boldsymbol{r}}  =  dy\wedge dz, \quad n_2 dS_{\boldsymbol{r}} = dz\wedge dx, \quad n_3 dS_{\boldsymbol{r}} = dx\wedge dy.
    	\end{equation}
    Therefore, the DLP \eqref{eq:introLapDLP} can be written as
    	\begin{equation} \label{eq:LapDLPext}
    		\begin{aligned}
    			\mathcal{D}[\mu](\boldsymbol{r}') &= \int_{\mathcal{M}} \frac{(\boldsymbol{r}'-\boldsymbol{r})\cdot \boldsymbol{n}_{\boldsymbol{r}}}{4\pi|\boldsymbol{r}'-\boldsymbol{r}|^3}\mu(\boldsymbol{r}) dS_{\boldsymbol{r}} \\
    			&= \int_{\mathcal{M}} \frac{(x'-x)\mu(\boldsymbol{r})}{4\pi|\boldsymbol{r}'-\boldsymbol{r}|^3} dy\wedge dz + \frac{(y'-y)\mu(\boldsymbol{r})}{4\pi|\boldsymbol{r}'-\boldsymbol{r}|^3} dz\wedge dx + \frac{(z'-z)\mu(\boldsymbol{r})}{4\pi|\boldsymbol{r}'-\boldsymbol{r}|^3} dx\wedge dy,
    		\end{aligned}
    	\end{equation}
    where $\boldsymbol{r}'=(x',y',z')$ and $\boldsymbol{r}=(x,y,z)$.

\section{Density approximation and exact form construction}
\label{Scheme}

In this section, we systematically introduce the key ideas required to develop our numerical scheme. We provide the necessary analytical and algebraic background employed in Section~\ref{NumerScheme}. We briefly review Stokes theorem and Poincar\'e's Lemma to illustrate our basic ideas in Section~\ref{subsec:s_thm&p_lem}, and introduce a quaternionic approximation scheme in Section~\ref{subsec:q_approx}. 
%Then later in Section~\ref{NumerScheme}, we summarize the problem setup, the discretization scheme, and the complete implementation of the proposed scheme. 

%==========================================================================
% \subsection{Theoretical background and quaternionic approximation}

%--------------------------------------------------------------------------
\subsection{Stokes theorem and Poincar\'e's lemma}
\label{subsec:s_thm&p_lem}
    We will be relying on the Stokes theorem to evaluate \eqref{eq:LapDLPext} when $\boldsymbol{r}'$ is close to $\mathcal{M}$. Using exterior calculus, one can summarize the Stokes theorem on a patch $D$ in an elegant way \cite{spivak2018calculus}:
    \begin{theorem} \label{thm:Stokes}	
     (Stokes theorem) If $D$ is a compact oriented 2-manifold, for any smooth 1-form $\omega$ defined on $D$, the following holds,
    \begin{equation}\label{eq:stokes_thm}
    		\begin{aligned}
    			\int_{D} d\omega = \int_{\partial D} \omega.
    		\end{aligned}
    	\end{equation}
    \end{theorem}	
    \begin{rmk}
    The advantage of using Stokes theorem to reduce a surface integral of $2$-form $d\omega$ on $D$ to a line integral of $1$-form $\omega$ is essentially two-fold. One is that we have localized the work involved in evaluating layer potential on part of the integration surface. The other comes from the benefit of dimensionality reduction. Essentially, this eliminates singularity that populates the two dimensional manifold to only boundaries of its panel discretization, which has a measure zero. This further helps in accurate evaluation of layer potentials when targets are extremely close to or on the surface. 
    \end{rmk} 	
	The key idea is to use Stokes theorem to evaluate the double-layer potential when a target $\boldsymbol{r}'$ is close to $D$. But Stokes theorem does not help with finding a suitable $\omega$ such that
	\begin{equation} \label{eq:omega_approx}
		\begin{aligned}
		\frac{(\vv{r}'-\vv{r})\cdot \vv{n}_{\vv{r}}}{4\pi|\vv{r}'-\vv{r}|^3}\mu(\vv{r}) dS_{\vv{r}} = 	\frac{(x'-x)\mu(\boldsymbol{r})}{4\pi|\boldsymbol{r}'-\boldsymbol{r}|^3} dy\wedge dz + \frac{(y'-y)\mu(\boldsymbol{r})}{4\pi|\boldsymbol{r}'-\boldsymbol{r}|^3} dz\wedge dx + \frac{(z'-z)\mu(\boldsymbol{r})}{4\pi|\boldsymbol{r}'-\boldsymbol{r}|^3} dx\wedge dy \approx d\omega.
		\end{aligned}
	\end{equation}	
    To address the challenges in systematically finding $\omega$, we introduce one additional tool in differential geometry, Poincar\'e's lemma \cite[Thm.~4-11]{spivak2018calculus}: for every differential form on an open star-shaped subset $D$ of $\mathbb{R}^n$, suppose $d\alpha=0$ for $\alpha\in \Lambda^k(D)$, then locally there is some $\omega\in\Lambda^{k-1}(D)$ such that $d\omega = \alpha$.  
    The proof of the lemma considers a $k-$form,
    	\begin{equation} \label{eq:alpha} \alpha = \sum_{i_1<\cdots<i_k} g_{i_1,\cdots,i_k} dx^{i_1}\wedge\cdots\wedge dx^{i_k}, 
    	\end{equation}
    and shows that $(k-1)$-form $\omega$ defined by
    	\begin{equation} \label{eq:omega}
    %	\[%\label{eq:poincare}
    		%\begin{aligned}
    			\omega =  \sum_{i_1<\cdots<i_k}\sum_{l=1}^k (-1)^{(l-1)} \left(\int_0^1 t^{k-1}g_{i_1,\cdots,i_k}(t\boldsymbol{x})dt\right)x^{i_l} \,
    %				 \hspace{1in} 
    				dx^{i_1}\wedge\cdots\wedge\hat{dx^{i_l}}\wedge\cdots\wedge dx^{i_k}
    		%\end{aligned}
    \end{equation}
    satisfies $d\omega = \alpha$ given $d\alpha=0$. 	
	
    Based on this result, we can accomplish the 2-to-1 form conversion as written in \eqref{eq:omega_approx}. This construction process can be viewed as finding the vector potential whose curl is a given vector field. 

    A simplified version of the Poincar\'e's Lemma relevant to our setting can be summarized as follows. 
	\begin{lemma}(2-to-1 form conversion)\label{lemma:potential}
	Consider a compact oriented $2$-dimensional manifold $D$ in $\mathbb{R}^3$. Let
		\begin{equation}
			\begin{aligned}
				\alpha = g_1(\boldsymbol{r})dy\wedge dz + g_2(\boldsymbol{r}) dz\wedge dx + g_3(\boldsymbol{r}) dx\wedge dy,
			\end{aligned}
		\end{equation}
		be a differential $2$-form on $D$.
		If $d\alpha = 0$ (i.e., $\nabla\cdot(g_1,g_2,g_3)=0$), then
		\begin{equation}\label{eq:poincare_r3}
			\begin{aligned}
				\omega =& \left(\int_0^1\left(tzg_2(t\boldsymbol{r})-tyg_3(t\boldsymbol{r})\right)dt\right) dx+\left(\int_0^1\left(txg_3(t\boldsymbol{r})-tzg_1(t\boldsymbol{r})\right)dt\right)dy\\
					&\hspace{2in}+\left(\int_0^1\left(tyg_1(t\boldsymbol{r})-txg_2(t\boldsymbol{r})\right)dt\right)dz
			\end{aligned}
		\end{equation}
		satisfies $d\omega=\alpha$.
	\end{lemma}
	\begin{proof}
		See Appendix~\ref{appx:verify_poin}
	\end{proof}
	
    Consequently, it is possible to convert a surface integral into a line integral as long as the vector $\vv{g}$ is divergence-free. For example, if $\mu$ is a scalar in  \eqref{eq:omega_approx}, it is clear that the above Lemma applies since 
    \[\nabla \cdot \frac{\vv{r}' - \vv{r}}{|\vv{r}' - \vv{r}|^3} = \nabla \cdot \nabla \frac{1}{|\vv{r}' - \vv{r}|} = 0.  
    \] % no negative sign
    In the general case, our goal is to find a high-order approximation scheme for $\mu$ which makes the vector $(\vv{r}' - \vv{r})\mu(\vv{r})/|\vv{r}' - \vv{r}|^3$ divergence-free. Clearly, standard polynomial approximation schemes (e.g., tensor-product monic polynomials) won't yield the desired result. In the next subsection, we present an approximation scheme based on harmonic polynomials and quaternionic representations that accomplishes this task.

    The key insight that motivates our approach is summarized in Lemma \ref{lemma:harmonic}, but first let's review some preliminaries on quaternions. Let $\vv{i}, \vv{j}$ and $\vv{k}$ be the standard quaternion units, that is, they satisfy the identities
    	\begin{equation}
    	    \vv{i}^2 = \vv{j}^2 = \vv{k}^2 = \vv{ijk} = -1, \quad \vv{ij} = \vv{k} = -\vv{ji}, \quad \vv{jk} = \vv{i} = -\vv{kj},\quad \vv{ki} = \vv{j} = -\vv{ik}.
    \label{quaternion}	\end{equation}
    A quaternionic  function $g$ is comprised of a scalar part $g_0$ and a vector part $\vv{g} = (g_1, g_2, g_3)$, and written as 
    	\[g(\vv{r}) = g_0(\vv{r}) + g_1(\vv{r}) \vv{i} + g_2(\vv{r}) \vv{j} + g_3(\vv{r}) \vv{k}. \]
    %while the vector function $\vv{g}(\vv{r}) = (g_1(\vv{r}), g_2(\vv{r}), g_3(\vv{r}))$. 
    Alternatively, one can write the quaternion in the pair form as $g = (g_0, \vv{g})$. 
    For any given vector $\vv{g}$, we can define a quaternion $g$ as above and if not specified, $g_0 = 0$ is assumed by default. 
    %On the other hand, when we associate a vector $\vv{g}$ to a quaternion $g$, we set $g_0 = 0$ by default, unless specified otherwise.
    Using \eqref{quaternion}, we can easily verify that the product of two quaternions $g$ and $h$ to be
    	\[ gh = (g_0 h_0 - \vv{g}\cdot\vv{h}, \, g_0\vv{h} + h_0\vv{g} + \vv{g} \times \vv{h}). \]
    With these preliminaries, we can now state and prove the following lemma that motivates our density approximation scheme. 
	\begin{lemma} \label{lemma:harmonic}
		Let $\vv{g} = \nabla \psi$ and $\vv{f} = \nabla \phi$, where $\phi$ and $\psi$ are some harmonic functions.  Then each component of the quaternionic 2-form $g n_r f dS_{\boldsymbol{r}}$ is exact. 
	\end{lemma}
	\begin{proof}
	Using the identities \eqref{quaternion} and after some algebra, we obtain
	    \[ gn_rf = \left(-\left(\vv{g}\times\vv{n}_{\vv{r}}\right)\cdot\vv{f}, -\left(\vv{g}\cdot\vv{n}_{\vv{r}}\right)\vv{f}+\left(\vv{g}\times\vv{n}_{\vv{r}}\right)\times\vv{f}\right) \]
        The scalar part of $g\, n_r\, f\, dS_{\boldsymbol{r}}$ then becomes
            \[ -\left(\vv{g}\times\vv{n}_{\vv{r}}\right)\cdot\vv{f}\, dS_{\boldsymbol{r}} = -\left(\vv{f}\times\vv{g}\right)\cdot\vv{n}_{\vv{r}}dS_{\vv{r}}\] % check
        Invoking the fact that both $\vv{g}$ and $\vv{f}$ are gradients of harmonic functions, we get            \[\nabla\cdot\left(\vv{f}\times\vv{g}\right) = \vv{g}\cdot\left(\nabla\times\nabla\phi\right)-\vv{f}\cdot\left(\nabla\times\nabla\psi\right)  = 0,\] % check
        thereby, confirming that the scalar part is an exact form (using Lemma \ref{lemma:potential}). Similarly, the first component of its vector part is given by
            \[\left(n_1\left(\vv{g}\cdot\vv{f}\right)-g_1\left(\vv{f}\cdot\vv{n}_{\vv{r}}\right)-f_1\left(\vv{g}\cdot\vv{n}_{\vv{r}}\right)\right)dS_{\boldsymbol{r}} = \left(\left(\vv{g}\cdot\vv{f}\right)\vv{e}_1-g_1\vv{f}-f_1\vv{g}\right)\cdot\vv{n}_{\vv{r}}dS_{\vv{r}}\] % check
        Therefore, this term is also exact since 
            \[\nabla\cdot\left(\left(\vv{g}\cdot\vv{f}\right)\vv{e}_1-g_1\vv{f}-f_1\vv{g}\right) = \frac{\partial}{\partial x}\left(\vv{g}\cdot\vv{f}\right)-\nabla g_1\cdot\vv{f}-g_1\Delta\phi-\nabla f_1\cdot\vv{g}-f_1\Delta\psi = 0.\] % check
	\end{proof}
%--------------------------------------------------------------------------
%
    The following corollary applies this Lemma to help verify the divergence-free condition in Lemma \ref{lemma:potential} to convert the double-layer integral into a 1-form integral. 
	\begin{corollary} \label{lemma:Laplace_kl}
       	Let $\alpha$ be a quaternionic differential $2$-form on $D$ given by
		\begin{equation}\label{eq:lm_2form_kl}
			\begin{aligned}
				\alpha = \alpha_0 + \alpha_1\vv{i} + \alpha_2\vv{j} + \alpha_3\vv{k} = \frac{(r'-r)n_{r}}{|r'-r|^3}f(\boldsymbol{r})dS_{\boldsymbol{r}}.  
			\end{aligned}
		\end{equation}
		If the vector $\vv{f}(\vv{r})$ is the gradient of a harmonic function, then $d \alpha_i=0$, $i=0,1,2,3$. 
		\end{corollary}
	\begin{proof}
		This result directly follows from Lemma \ref{lemma:harmonic} by setting $\vv{g}(\vv{r}) = \nabla \frac{1}{|\vv{r}' - \vv{r}|}$ and noting that $|r'-r|=|\boldsymbol{r}' - \vv{r}|$.
	\end{proof}
%--------------------------------------	
\subsection{Approximation scheme using harmonic polynomials and quaternionic representation}
\label{subsec:q_approx}		
    Our goal is to express the density $\mu$ in terms of some basis functions that allow us to apply Corollary \ref{lemma:Laplace_kl} to convert the DLP  \eqref{eq:LapDLPext} to a 1-form using Lemma \ref{lemma:potential}. From Corollary \ref{lemma:Laplace_kl}, it is clear that the elements of such a basis set essentially must be in quaternionic form and their vector components must be gradients of some harmonic functions (notice that in Corollary \ref{lemma:Laplace_kl} we need to expand the quaternionic form $\frac{(r'-r)n_{r}}{|r'-r|^3}f(\boldsymbol{r})dS_{\boldsymbol{r}}$ before projecting the surface elements).	To construct such a basis set, we turn to {\em harmonic polynomials}, which are homogenous polynomials that satisfy the Laplace equation (we refer the reader to \cite{gallier2009notes} for a review on this topic). 
	
	While there are $2p+1$ independent harmonic polynomials of degree $p$, we will chose $p$ of them for our construction. The only requirement is that their gradients must be linearly independent. We use, for example, the following set of harmonic polynomials for upto degree 7:
	\begin{equation} \label{eq:harmonics}
		\begin{aligned}
			\mathcal{P}_1= & \{z\}, \ \mathcal{P}_2= \{x^2-z^2, \ y^2-z^2\},\ \mathcal{P}_3= \{x^3-3xz^2, \ y^3-3yz^2, \ xyz\},\\
			\mathcal{P}_4= &\{x^4-6x^2z^2+z^4, \ y^4-6y^2z^2+z^4, 3x^2yz-yz^3, \ 3xy^2z-xz^3\},\\
    		\mathcal{P}_5= &\{x^5-10x^3z^2+5xz^4, \ y^5-10y^3z^2+5yz^4, \ x^4y-6x^2yz^2+yz^4,\\
    						& \hspace{0.1in} xy^4-6xy^2z^2+xz^4,\ -15x^2y^2z+5x^2z^3+5y^2z^3-z^5\},\\
    		\mathcal{P}_6= &\{x^6-15x^4z^2+15x^2z^4-z^6, \ y^6-15y^4z^2+15y^2z^4-z^6, \\
							& \hspace{0.1in} x^5y-10x^3yz^2+5xyz^4, \ xy^5-10xy^3z^2+5xyz^4\\
							& \hspace{0.1in} 5x^4yz-10x^2y^3z+y^5z, \ 5xy^4z-10x^3y^2z+x^5z\},\\
    		\mathcal{P}_7= &\{x^7-7xz^6+35x^3z^4-21x^5z^2, \ y^7-7yz^6+35y^3z^4-21y^5z^2, \\
    							& \hspace{0.1in} x^6y-15x^4yz^2+15x^2yz^4-z^6y, \ xy^6-15xy^4z^2+15xy^2z^4-xz^6, \\
    							& \hspace{0.1in} 3x^5y^2-3x^5z^2-30x^3y^2z^2+10x^3z^4+15xy^2z^4-3xz^6,\\
    							& \hspace{0.1in} 3x^2y^5-3y^5z^2-30x^2y^3z^2+15x^2yz^4+10y^3z^4-3yz^6,\\
    							& \hspace{0.1in} (3x^5y-10x^3y^3+3xy^5)z\}.
		\end{aligned}
	\end{equation}
	
    Then, we assign the gradients of each of these harmonic polynomials, expressed in quaternionic form, as the elements of the required basis set. Denoting this set by $\{f^{(k,1)},\cdots, f^{(k,k)}\}, k=1,\cdots,p$, we can easily derive them from \eqref{eq:harmonics} as
    	\begin{equation}
    		\begin{aligned}
    			&\nabla \mathcal{P}_1 = \{f^{(1,1)}= \vv{k}\},  \quad  \nabla \mathcal{P}_2 = \{f^{(2,1)}=x\vv{i}-z\vv{k},\ f^{(2,2)}=y\vv{j}-z\vv{k} \}, \\
    			\nabla \mathcal{P}_3 = \{ f^{(3,1)}&=\left(x^2-z^2\right)\vv{i}-2xz\vv{k},\,\, f^{(3,2)}=\left(y^2-z^2\right)\vv{j}-2yz\vv{k}, \,\, f^{(3,3)}=yz\vv{i}+xz\vv{j}+xy\vv{k} \},\end{aligned}
    	\end{equation}
    and so on. Therefore, in total, there are $p(p+1)/2$ quaternionic functions in this basis set. Moreover, the set $\nabla \mathcal{P}_p$ is composed of homogenous, quaternionic polynomials of degree $(p-1)$. Thereby, a $p^{\text{th}}$ order convergent scheme is obtained when 
    the set $\{\nabla \mathcal{P}_k, \, k = 1, \ldots, p \}$ is used for approximating smooth quaternionic functions. 

    Now, consider a triangular patch $D \subset \mathcal{M}$, as illustrated in Fig.~\ref{fig:geom}(b). On this patch, we use the basis functions $\{ f^{(k,l)}(\boldsymbol{r})\}$ to approximate the density function as 
    	\begin{equation} \label{eq:q_approx}
    		\begin{aligned}
        	\mu(\boldsymbol{r})+0\vv{i}+0\vv{j}+0\vv{k} \approx 
        	\sum_{k=1}^p\sum_{l=1}^k f^{(k,l)}(\boldsymbol{r})\,c^{(k,l)},
        	\end{aligned}
        \end{equation}
    where the unknown coefficients $c^{(k,l)}$ are also quaternions, that is, 
    \[c^{(k,l)} = c^{(k,l)}_0 + c^{(k,l)}_1\vv{i}+c^{(k,l)}_2\vv{j}+c^{(k,l)}_3\vv{k}.\] %
    Therefore, there are $4\cdot \frac{p(p+1)}{2}$ unknowns that need to be determined. They can be obtained, for instance, by applying a standard collocation scheme. For a $p$-node composite tensor product Gauss-Legendre quadrature in Fig.~\ref{fig:geom}(c), there are $(p^2+p)/2$ quadrature nodes on the pre-image of each triangular patch $D$, denoted as $\{\boldsymbol{r}^{(k,l)} | 1\leq k\leq l\leq p\}$. Enforcing \eqref{eq:q_approx} at these quadrature nodes will generate the required $4\cdot \frac{p(p+1)}{2}$ number of equations. Empirically, we found that the resulting square linear systems are invertible in general. 

    \begin{rmk}    
    In our implementation, we perform a change of coordinates in each patch so that  $\boldsymbol{r}^{(1,1)}$ becomes the origin (as illustrated in Fig.~\ref{fig:geom}(c)). In this case, $c^{(1,1)}=\mu(\boldsymbol{r}^{(1,1)})$ is known, and there are only $(p^2+p)/2-1$ quaternionic unknowns (and corresponding equations). \end{rmk}

    We are now ready to substitute the density approximated as in \eqref{eq:q_approx} into the DLP \eqref{eq:LapDLPext}. However, the double-layer kernel needs to be written in quaternionic form to take advantage of Corollary \ref{lemma:harmonic}. Following lemma summarizes the result. 
		\begin{lemma}\label{lemma:LaplaceQuart}
 			Let $r$ and $n_r$ be the source location and the normal on $D$ respectively, in quaternionic form:  
 				\begin{equation}
                         \begin{aligned}
                                 r = 0 + x\vv{i} + y\vv{j} + z\vv{k}, \ n_r = 0 + n_1\vv{i} + n_2\vv{j} + n_3\vv{k}.
                        \end{aligned}
                 \end{equation}
    If the density is approximated as in \eqref{eq:q_approx}, then the scalar part of
           		\begin{equation}\label{eq:lm_dlp_approx}
                        \begin{aligned}
                                -\int_{D} \frac{(r'-r)n_{r}}{4\pi|r'-r|^3}\left(\sum_{k=1}^p\sum_{l=1}^k f^{(k,l)}(\boldsymbol{r})c^{(k,l)}\right)dS_{\boldsymbol{r}},
                        \end{aligned}
                \end{equation}
          	is a $p^\text{th}$ order convergent scheme to the DLP defined on $D$.
		\end{lemma}
		\begin{proof}
			The proof follows from the following two observations: (i) owing to \eqref{eq:q_approx}, only the scalar part of $\sum_{k=1}^p\sum_{l=1}^k f^{(k,l)}c^{(k,l)}$ remains and (ii) the scalar part of quaternion product $rn_r$ is given by $-\boldsymbol{r}\cdot \boldsymbol{n}_{\boldsymbol{r}}$. 
		\end{proof}
\begin{figure}
        \centering
        \includegraphics[width=\linewidth]{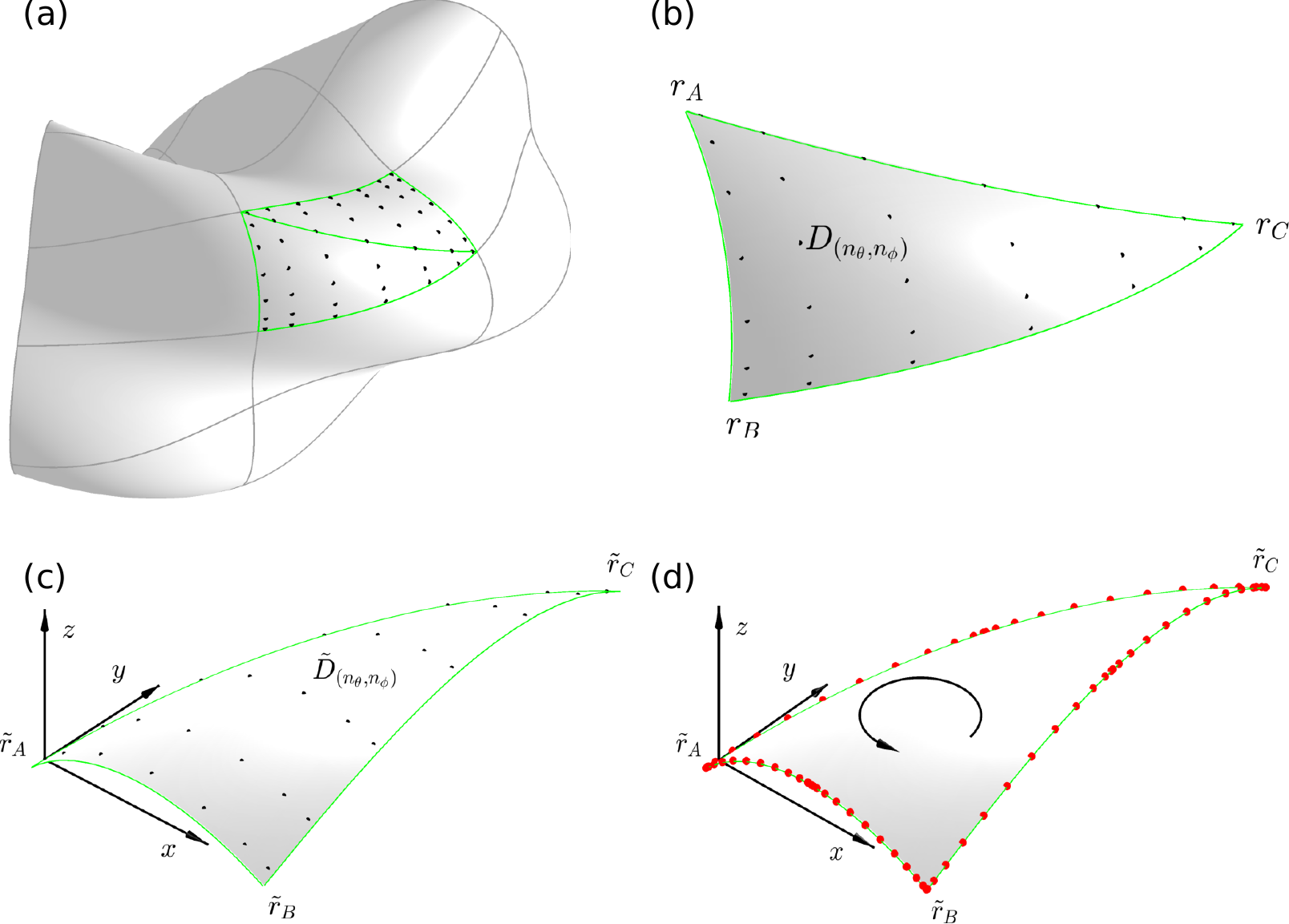}
         \caption{Schematic of our proposed product integration scheme for Laplace double-layer potential. (a) is part of a parametrized ``cruller'' surface. (b) is one of the triangular sub panel, denoted by $D_{(n_{\theta},n_{\phi})}$, from the rectangular panel in (a). (c) is the transformed triangular patch $\tilde{D}_{(n_{\theta},n_{\phi})}$. (d) shows the quadrature nodes (red) on integration contour $\partial \tilde{D}_{(n_{\theta},n_{\phi})}$, the boundary of the transformed triangular patch.%
    }\label{fig:geom}
\end{figure}

    Finally, to construct basis functions upto an arbitrary order beyond $p=7$, we can exploit the fact that restricting harmonic polynomials to a unit sphere yields the standard spherical harmonics \cite{gallier2009notes}. Similar to \eqref{eq:harmonics}, we can chose $p$ of the $2p+1$ solid spherical harmonics of degree $p$. For example, let $(\rho, \theta, \phi)$ be the coordinates of a point in spherical coordinates, then, we can set 
	\begin{equation}\label{eq:q_approx_higher}
		\begin{aligned}
			\mathcal{P}_p = \{\rho^p\cos(k\phi) P^k_p(\cos\theta), \, k=1,\cdots,p\},
		\end{aligned}
	\end{equation} 
	where $P^k_p(\cos\theta) = (-1)^k\sin^k\theta \frac{d^k}{d\left(\cos\theta\right)^k}\left(P_p(\cos\theta)\right)$ are the associated Legendre polynomials. When these functions are expressed in terms of the Cartesian coordinates, $x = \rho \sin{\theta}\cos{\phi}, y = \rho \sin{\theta}\sin{\phi}$ and $z = \rho \cos{\theta}$, we get a set of homogeneous harmonic polynomials of degree $p$. For example, consider the case of $p=8$, $k=2$, we get the following after using basic trigonometric identities:
	\begin{equation}\label{eq:q_approx_8th}
		\begin{aligned}
		\rho^8\cos(2\phi)P_8^2(\cos\theta) =\frac{315}{16}\rho^8\cos(2\phi)\sin^2\theta\left(143\cos^6\theta-143\cos^4\theta+33\cos^2\theta-1\right)&\\ 
		= \frac{315}{16}\left(x^2-y^2\right)
		\left(143z^6-143z^4(x^2+y^2+z^2)+  33z^2(x^2+y^2+z^2)^2- \right.& \left.(x^2+y^2+z^2)^3\right).
	    \end{aligned}
	\end{equation}
	The convergence results of the approximation using this basis upto order 10 are shown in Fig.~\ref{fig:numer_exp1}. 

\section{Numerical scheme}
\label{NumerScheme}
We now have all the tools necessary to build a high-order accurate close evaluation scheme. Here, we will restrict our discussion to toroidal geometries parameterized by an infinitely-differentiable, doubly-periodic function $\boldsymbol{r}(\theta,\phi)=\left(x(\theta,\phi), y(\theta,\phi), z(\theta,\phi)\right)$, where $(\theta,\phi) \in [0,2\pi)^2$, with the understanding that the scheme can be generalized to other topologies and problem settings (since it inherently works at the level of local patches).  Such a surface can be covered by the images of disjoint union of uniform rectangular patches in parameter space as shown in Fig.~\ref{fig:geom}(a). When target $\boldsymbol{r}'$ is far from $\mathcal{M}$, we simply use the standard Nystr\"om discretization for evaluation of \eqref{eq:introLapDLP} based on a composite tensor-product Gauss-Legendre quadrature (e.g., see \cite{barnett2019high}).

For both weakly-singular and nearly-singular integrals, our numerical scheme essentially remains the same. Therefore, from here on, we do not distinguish whether the target is on or off the surface. A high-level schematic description of the scheme is given in Fig.~\ref{fig:geom}(b-d). In a patch $D$ where the double-layer potential is singular, we reduce the 2-form integration to 1-form integration on $\partial D$. The steps involved in this conversion are described next.

\subsection{Close evaluation scheme for Laplace double-layer potentials} \label{sc:LapDeval}
	The computation proceeds in two stages. In the first stage, which is {\em target-independent}, the density function is expressed in terms of the quaternionic basis functions $\{ \nabla \mathcal{P}_k\}$. In the second stage, which is {\em target-dependent}, a sequence of steps are outlined that accomplish 2-form to 1-form conversion of the DLP. 

	\paragraph{Stage 1: Precomputation} Given the spatial discretization parameters, $n_{\theta}$ and $n_{\phi}$, we cover $\mathcal{M}$ with $n_{\theta}$-by-$n_{\phi}$ rectangular patches, with a composite $p$-by-$p$ Gauss-Legendre quadrature on each patch as shown in Fig.~\ref{fig:geom}(a). On a standard patch, we further divide it into two triangular patches as shown in Fig.~\ref{fig:geom}(a). For each triangular patch $D$, we first identify the set of all close targets. A coordinate transform is applied on each such target $\boldsymbol{r}'$ and the sources in $D$ such that  $\boldsymbol{r}^{(1,1)}$ becomes the origin (Fig.~\ref{fig:geom}(c)). Points after transformation are denoted as $\tilde{\boldsymbol{r}}'$, and $\tilde{\boldsymbol{r}}^{(k,l)}$, $1\leq l\leq k\leq p$. 
		
	Enforcing \eqref{eq:q_approx} in the transformed coordinate system at the rest of the quadrature nodes, we obtain the following system of vector equations: 
	\begin{equation}\label{eq:quaternion_algo}
		\begin{aligned}
			\sum_{k=2}^{p}\sum_{l=1}^{k}A[f^{(k,l)}](\tilde{\boldsymbol{r}}^{(i,j)})C^{(k,l)} = U^{(i,j)},\quad 1\leq j\leq i, \, 1<i\leq p,
		\end{aligned}
	\end{equation}
	where the operator $A[\cdot]$ acting on a quaternionic function $f$, the unknown coefficient vector $C$ and the right-hand-side vector $U$ are given by
    \begin{equation}\label{eq:quaternion_xmatrix}
         	A[f] = 
				\begin{pmatrix}
					0 & -f_1 & -f_2 & -f_3 \\
					f_1 & 0 & -f_3 & f_2\\
					f_2 & f_3 & 0 & -f_1\\
					f_3 & -f_2 & f_1 & 0
					\end{pmatrix},\, C^{(k,l)} = \begin{pmatrix} c_0\\ c_1\\ c_2\\ c_3
                \end{pmatrix}^{(k,l)},
    \,\text{and}\quad U^{(i,j)} = \begin{pmatrix} \mu(\boldsymbol{r}^{(i,j)})-\mu(\boldsymbol{r}^{(1,1)}) \\ 0 \\0 \\0 \end{pmatrix}.
    \end{equation}
    Assembling the vector equations \eqref{eq:quaternion_algo} for all $(i,j)$ yields a square linear system of size $4\cdot\left(\frac{p(p+1)}{2} - 1\right)$, for which we simply apply a direct solver. Therefore, this stage incurs a computational cost of $\mathcal{O}(p^6)$ per triangular patch $D$.   

%    \begin{rmk}
    % Unlike layer potential evaluation for far targets, which uses grid of rectangular patches, evaluation at close targets requires further division of these rectangular patches to two triangular patches as shown in Fig.~\ref{fig:geom}(a). Density $\mu$ approximation on these two triangular patches were computed separately and never were they combined together. But the authors would like to point out that in all of our numerical experiments, we retain the original rectangular patches considering the boundary line integrals of $1$-forms along the shared boundary of these two triangular patches (diagonal of rectangular patch) are very close. %if they share a common reference point. 
    % \end{rmk}    
	\paragraph{Stage 2: 2-to-1 form conversion and contour integration} Once the quaternionic coefficients $c^{(k,l)}$ for approximating the density are found in Stage 1, the next step is to substitute them in \eqref{eq:lm_dlp_approx} and proceed with converting this 2-form to quaternionic differential $1$-form $\omega^{(k,l)}$ using Lemma \ref{lemma:potential}. After the coordinate transformation, we can rewrite \eqref{eq:lm_dlp_approx} as (we omit \textasciitilde{} on $\tilde{\vv{r}}$):
	\begin{equation}
	    \sum_{k=2}^p\sum_{l=1}^k \left(\frac{1}{4\pi}\int_{\tilde{D}} \alpha^{(k,l)}\right)c^{(k,l)} + \frac{1}{4\pi}\int_{\tilde{D}}\alpha^{(1,1)}\mu(\vv{r}^{(1,1)}), \quad\text{where} \quad \alpha^{(k,l)} = - \frac{(r'-r)n_{r}}{|r'-r|^3}\, f^{(k,l)} dS_{\boldsymbol{r}}.
	\end{equation}
    The quaternionic components of $\alpha^{(k,l)}$ can be simplified to the following after carrying out the product of quaternions on the right hand side:
	\begin{equation}
		\begin{aligned}
			\alpha^{(k,l)}_i =  \frac{1}{|\boldsymbol{r}'-\boldsymbol{r}|^3} \boldsymbol{q}^{(k,l)}_{i} \cdot \boldsymbol{n}_{\boldsymbol{r}} \, dS_{\boldsymbol{r}}, \ k=2,\cdots, p, \ 1\leq l \leq k,
		\end{aligned}
	\label{eq:alpha_i}\end{equation}
    where $i=0,1,2,3$ corresponds is the index of the quaternion and each of the vectors $\boldsymbol{q}^{(k,l)}_{i}$ are $k^{\text{th}}$ degree polynomials in $\boldsymbol{r}$, which can be derived using quaternion product as
	\begin{equation}
		\begin{aligned}
		    &\vv{q}^{(k,l)}_0(\vv{r}',\vv{r}) = \vv{f}^{(k,l)}(\vv{r})\times \left(\vv{r}'-\vv{r}\right),\\ 
		    &\vv{q}^{(k,l)}_1(\vv{r}',\vv{r}) = -\left(\left(\vv{r}'-\vv{r}\right)\cdot\vv{f}^{(k,l)}(\vv{r})\right)\vv{e}_1 + \left(x'-x\right)\vv{f}^{(k,l)}(\vv{r}) + f^{(k,l)}_1(\vv{r})\left(\vv{r}'-\vv{r}\right),\\
		    &\vv{q}^{(k,l)}_2(\vv{r}',\vv{r}) = -\left(\left(\vv{r}'-\vv{r}\right)\cdot\vv{f}^{(k,l)}(\vv{r})\right)\vv{e}_2 + \left(y'-y\right)\vv{f}^{(k,l)}(\vv{r}) + f^{(k,l)}_2(\vv{r})\left(\vv{r}'-\vv{r}\right),\\
		    &\vv{q}^{(k,l)}_3(\vv{r}',\vv{r}) = -\left(\left(\vv{r}'-\vv{r}\right)\cdot\vv{f}^{(k,l)}(\vv{r})\right)\vv{e}_3 + \left(z'-z\right)\vv{f}^{(k,l)}(\vv{r}) + f^{(k,l)}_3(\vv{r})\left(\vv{r}'-\vv{r}\right).
		\end{aligned}
	\label{eq:qi} \end{equation}
				
    The advantage of our density approximation scheme is now apparent: from \eqref{eq:alpha_i}, it is clear that $\alpha_i^{(k,l)}$ is an exact form by construction and  we can apply Lemma \ref{lemma:potential} to convert it into a 1-form $\omega^{(k,l)}_i$ as follows:  
	\begin{equation} \label{eq:potential_coeff}
		\begin{aligned}
			\omega^{(k,l)}_i = &\left(\int_0^1 \frac{tzq^{(k,l)}_{i,2}(\boldsymbol{r}',t\boldsymbol{r})- tyq^{(k,l)}_{i,3}(\boldsymbol{r}',t\boldsymbol{r})}{|t\vv{r}-\vv{r}'|^3} dt \right) dx + \left(\int_0^1 \frac{txq^{(k,l)}_{i,3}(\boldsymbol{r}',t\boldsymbol{r})- tzq^{(k,l)}_{i,1}(\boldsymbol{r}',t\boldsymbol{r})}{|t\vv{r}-\vv{r}'|^3} dt \right) dy \\
			&+ \left(\int_0^1 \frac{tyq^{(k,l)}_{i,1}(\boldsymbol{r}',t\boldsymbol{r})- txq^{(k,l)}_{i,2}(\boldsymbol{r}',t\boldsymbol{r})}{|t\vv{r}-\vv{r}'|^3} dt \right) dz
		\end{aligned}
	\end{equation} 	
	Notice that the numerators of each of the integrands in the 1-form above are polynomials of degree $(k+1)$ in the variable $t$. Defining 
	$M_k(\boldsymbol{r}',\boldsymbol{r})= \int_0^1 \frac{t^k}{|t\vv{r}-\vv{r}' |^{3}}dt$ and using \eqref{eq:qi}, we can separate out the terms that depend on $t$ and rewrite $\omega^{(k,l)}_i$ as
	\begin{equation} \label{eq:1-form}
		\begin{aligned}
			\omega^{(k,l)}_i = &\left(v^{(k,l)}_{i,1}(\boldsymbol{r})M_{k+1}+w^{(k,l)}_{i,1}(\boldsymbol{r}',\boldsymbol{r})M_{k}\right) dx  + \left(v^{(k,l)}_{i,2}(\boldsymbol{r})M_{k+1}+w^{(k,l)}_{i,2}(\boldsymbol{r}',\boldsymbol{r})M_{k}\right) dy \\
				& + \left(v^{(k,l)}_{i,3}(\boldsymbol{r})M_{k+1}+w^{(k,l)}_{i,3}(\boldsymbol{r}',\boldsymbol{r})M_{k}\right) dz,
		\end{aligned}
	\end{equation}
	where, for brevity, we omitted $(\boldsymbol{r}',\boldsymbol{r})$ dependency on $M_k$. Here, $v^{(k,l)}_{i,j}(\boldsymbol{r})$ are $(k+1)^{\text{th}}$ degree polynomials in $\boldsymbol{r}$ and $w^{(k,l)}_{i,j}(\boldsymbol{r}',\boldsymbol{r})$ are $k^{\text{th}}$ degree polynomials in $\boldsymbol{r}$ and linear in $\vv{r}'$. The explicit expressions for these can be easily derived from \eqref{eq:qi} and \eqref{eq:potential_coeff}; for example, in the case of $i=0$, we have 
	\begin{equation}
	    \begin{aligned}
	        &v^{(k,l)}_{0,1}(\vv{r}) = \left(y^2+z^2\right)f^{(k,l)}_1(\vv{r})-xyf^{(k,l)}_2(\vv{r})-xzf^{(k,l)}_3(\vv{r})\\
	        &w^{(k,l)}_{0,1}(\vv{r}',\vv{r}) = -\left(y'y+z'z\right)f^{(k,l)}_1(\vv{r})+x'yf^{(k,l)}_2(\vv{r})+x'zf^{(k,l)}_3(\vv{r})\\
	        &v^{(k,l)}_{0,2}(\vv{r}) = -xyf^{(k,l)}_1(\vv{r})+\left(x^2+z^2\right)f^{(k,l)}_2(\vv{r})-yzf^{(k,l)}_3(\vv{r})\\
	        &w^{(k,l)}_{0,2}(\vv{r}',\vv{r}) = y'xf^{(k,l)}_1(\vv{r})-\left(x'x+z'z\right)f^{(k,l)}_2(\vv{r})-y'zf^{(k,l)}_3(\vv{r})\\
	        &v^{(k,l)}_{0,3}(\vv{r}) = -xzf^{(k,l)}_1(\vv{r})-yzf^{(k,l)}_2(\vv{r})+\left(x^2+y^2\right)f^{(k,l)}_3(\vv{r})\\
	        &w^{(k,l)}_{0,3}(\vv{r}',\vv{r}) = z'xf^{(k,l)}_1(\vv{r})+z'yf^{(k,l)}_2(\vv{r})-\left(x'x+y'y\right)f^{(k,l)}_3(\vv{r})
	    \end{aligned}
	\end{equation}
	Also see Appendix \ref{appx:dlp_2nd_order} where we derive explicit formulas for all the quantities $\omega^{(k,l)}_i$, $v^{(k,l)}_{i,j}$ and $w^{(k,l)}_{i,j}$ in the case of a second-order scheme. 
	
	Therefore, the kernel singularity in the DLP is now encoded into the moments $\{M_k\}$. For each source $\boldsymbol{r}$ and target $\boldsymbol{r}'$, we evaluate these moments analytically via recurrences given in Appendix~\ref{appx:LMN}. Now, we can express the complete $1$-form $\omega$ as a linear combination of each individual $1$-forms with coefficients $C^{(k,l)}$ and the constant term corresponding to $k=1$ as 
	\begin{equation}
		\begin{aligned}
			\omega = \mu(\vv{r}^{(1,1)})\omega_0^{(1,1)} + \sum_{k=2}^p\sum_{l=1}^k\Omega^{(k,l)}C^{(k,l)}, 
		\end{aligned}
	\end{equation}
	where $\Omega^{(k,l)}=[\omega_0^{(k,l)},\omega_1^{(k,l)},\omega_2^{(k,l)},\omega_3^{(k,l)}]$. Then, we evaluate boundary path integral $\int_{\partial \tilde{D}} \omega$ using a high-order smooth quadrature rule (Gauss-Legendre). However, note that the 1-forms \eqref{eq:1-form} are still singular if $\vv{r}'$ approaches $\vv{r}$ (as can be inferred from the base conditions \eqref{eq:base}). But since the sources now reside on $\partial D$, this situation can be avoided altogether by simply choosing a larger patch.    
	
	The computational complexity of this stage is primarily dictated by the evaluation of the moments $\{M_k\}$. For each target, evaluating the recurrences for these moments take $\mathcal{O}(p)$ time for each of the $\mathcal{O}(p)$ sources on $\partial D$, bringing the total per-target cost to $\mathcal{O}(p^2)$. Therefore, the computational cost of both stages of the close evaluation per patch $D$ is $\mathcal{O}(p^6 + n_{targ} p^2)$, where $n_{targ}$ is the number of targets that are considered to be close to $D$.

\subsection{A generalization to evaluate the single-layer and the gradient of double-layer Laplace potentials}
\label{subsec:generalization}
    In several applications, one needs to evaluate high-order derivatives of layer potentials either as a post-processing step or, in some cases, as part of representing the solution itself. Lemma \ref{lemma:harmonic} simplifies this task for us since we can exploit the fact that the kernels themselves are harmonic functions. For example, assuming the density is approximated as in \eqref{eq:q_approx}, it is easy to show that the scalar part of the integral
	\begin{equation}\label{eq:graddlp}
		\begin{aligned}
          -\int_{\mathcal{M}}\frac{1}{|r'-r|^3}
           &\left(\left(\delta_{1j}-3\frac{(x'-x)(r'-r) }{|r'-r|^2}\right)n_{r},\
            \left(\delta_{2j}-3\frac{(y'-y)(r'-r)}{|r'-r|^2}\right)n_{r},\right.\\
            &\left.\hspace{0.5in} \left(\delta_{3j}-3\frac{(z'-z)(r'-r)}{|r'-r|^2}\right)n_{r}\right)\left(\sum_{k=1}^p\sum_{l=1}^k f^{(k,l)}(\boldsymbol{r})c^{(k,l)}\right)dS_{\boldsymbol{r}},
		\end{aligned}
	\end{equation}
    is an approximation to $\nabla\mathcal{D}(\mu)(\boldsymbol{r})$. The requirements $d\alpha=0$ for converting the directional derivatives of the double-layer potential to $1$-form $\omega$ are satisfied based on the interchangeability of partial derivatives and exterior derivatives. However, constructing the 1-forms in this case also requires us to evaluate integrals of the form $L_k(\boldsymbol{r}',\boldsymbol{r})= \int_0^1 \frac{t^k}{|t\vv{r}-\vv{r}' |^{5}}dt$; recurrences to evaluate these are provided in Appendix~\ref{appx:LMN}. 

 	Lastly, we consider the single-layer potential (SLP), given by $\mathcal{S}[\mu](\boldsymbol{r}') = \int_{\mathcal{M}} \frac{1}{|\vv{r}' - \vv{r}|} \mu(\boldsymbol{r})\ dS_{\boldsymbol{r}}$. Since the SLP does not have a $2$-form structure like the DLP as in \eqref{eq:omega_approx}, we instead write it as 
	\begin{equation} \label{eq:modifiedSLP}
			\begin{aligned}
				\mathcal{S}[\mu](\boldsymbol{r}')   = \int_{\mathcal{M}} \frac{\left( r'-r\right) n_{r}}{|r'-r|^3} \overline{n}_{r} \overline{\left(r'-r\right)} \mu(\boldsymbol{r}) \ dS_{\boldsymbol{r}},
			\end{aligned}
	\end{equation}
	using quaternion algebra. Now, we can treat it the same way as a DLP but with a modified density function. That is, we construct the following approximation,
	\begin{equation} \label{eq:q_approxSLP}
		\begin{aligned}
			\overline{n}_{r} \overline{\left(r'-r\right)} \mu(\boldsymbol{r}) \approx 
			\sum_{k=1}^p\sum_{l=1}^k f^{(k,l)}(\boldsymbol{r})\,c^{(k,l)},
		\end{aligned}
	\end{equation}
    substitute in \eqref{eq:modifiedSLP} and follow the close evaluation scheme as in Section \ref{sc:LapDeval}. While it appears like this approach requires us to perform the approximation \eqref{eq:q_approxSLP} independently for each target, since the term $\overline{\left(r'-r\right)}$ is separable, we can just form it once per patch with minor bookkeeping. 

\section{Numerical results and discussion}
\label{NumExp} 
In this section, we present numerical results from a series of tests to validate the accuracy of our close evaluation scheme. We consider three different geometries of toroidal and spherical topologies. First, we conduct a convergence test of the quaternionic approximation algorithm presented in Section~\ref{subsec:q_approx}. We then test the performance of our DLP close evaluation, first {\em via} self-convergence on three different geometries, followed by a boundary value problem solve, whose exact solution is known analytically. 

	\subsection{Convergence properties of the quaternionic approximation}
	\label{subsec:numer_exp1}
	Consider a local patch $D$ and a smooth function $\mu$ defined on it. We can express the coordinate functions of $D$ and $\mu$ in terms of the tensor-product monic polynomials:
	\begin{equation}
		\begin{aligned}
			\boldsymbol{r} \approx [x,y,z]^T = [x,y,\sum_{1\leq k+l\leq m} a^{D}_{k,l}x^ky^l]^T,\quad \mu(\boldsymbol{r}) \approx \sum_{1\leq k'+l'\leq m} a^{\mu}_{k',l'}x^{k'}y^{l'},
		\end{aligned}
	\label{taylor}\end{equation} 
	where $a^{D}_{k,l}$ and $a^{\mu}_{k',l'}$ are the Taylor coefficients of the $z-$component and $\mu$ respectively, in terms of $x$ and $y$. 

	In this test, we provide empirical evidence that the quaternionic  approximation scheme employed in \eqref{eq:q_approx} is $p^{\text{th}}$ order convergent. To do so, we proceed as follows. In the domain $(x,y) \in [-0.5, 0.5]^2$, we generate a sequence of the function pairs $(z, \mu)$ by setting $a^{D}_{k,l} = a^{\mu}_{k',l'} = 1$ and the rest to zero for each admissible pairs of indices (i.e., such that $1\leq k+l\leq m$ and $1\leq k'+l'\leq m$). Two such pairs are plotted in Fig. \ref{fig:numer_exp1} (left). 
	 Then, on each of these pairs, we construct the density approximation in quaternionic form \eqref{eq:q_approx} {\em via} collocation by solving \eqref{eq:quaternion_algo}. We report the relative $L^{\infty}$ error between $\mu(\vv{r})$ and $\sum_{k=1}^p\sum_{l=1}^k f^{(k,l)}(\vv{r})\,c^{(k,l)}$ measured on a $250\times 250$ target grid for four different $h_D$, leg size of right triangular patches, in Fig. \ref{fig:numer_exp1} (middle). We observe that $9$-digit accuracy is reached as we reduce $h_D$. Moreover, from Fig. \ref{fig:numer_exp1} (right), we can see that the expected order of convergence is attained asymptotically. 
		
	\begin{figure}[h!]
        \centering
        \includegraphics[height=0.28\textwidth]{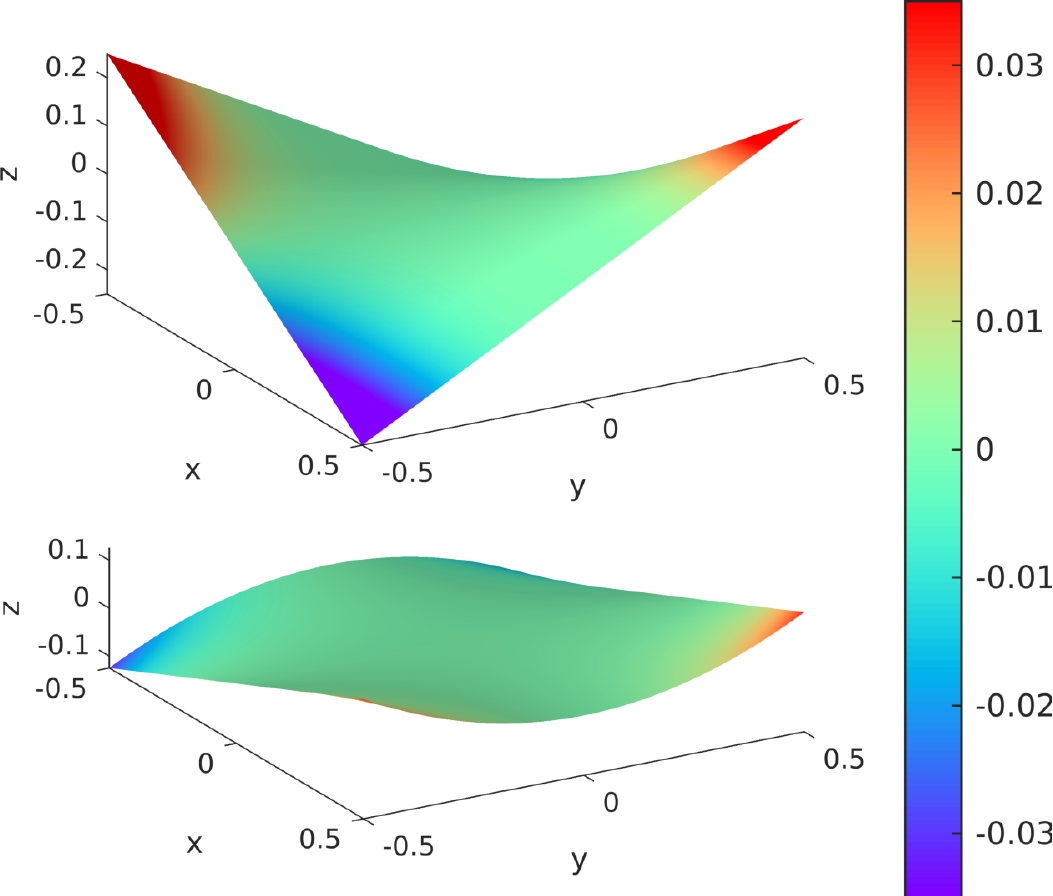}
        \includegraphics[height=0.28\textwidth]{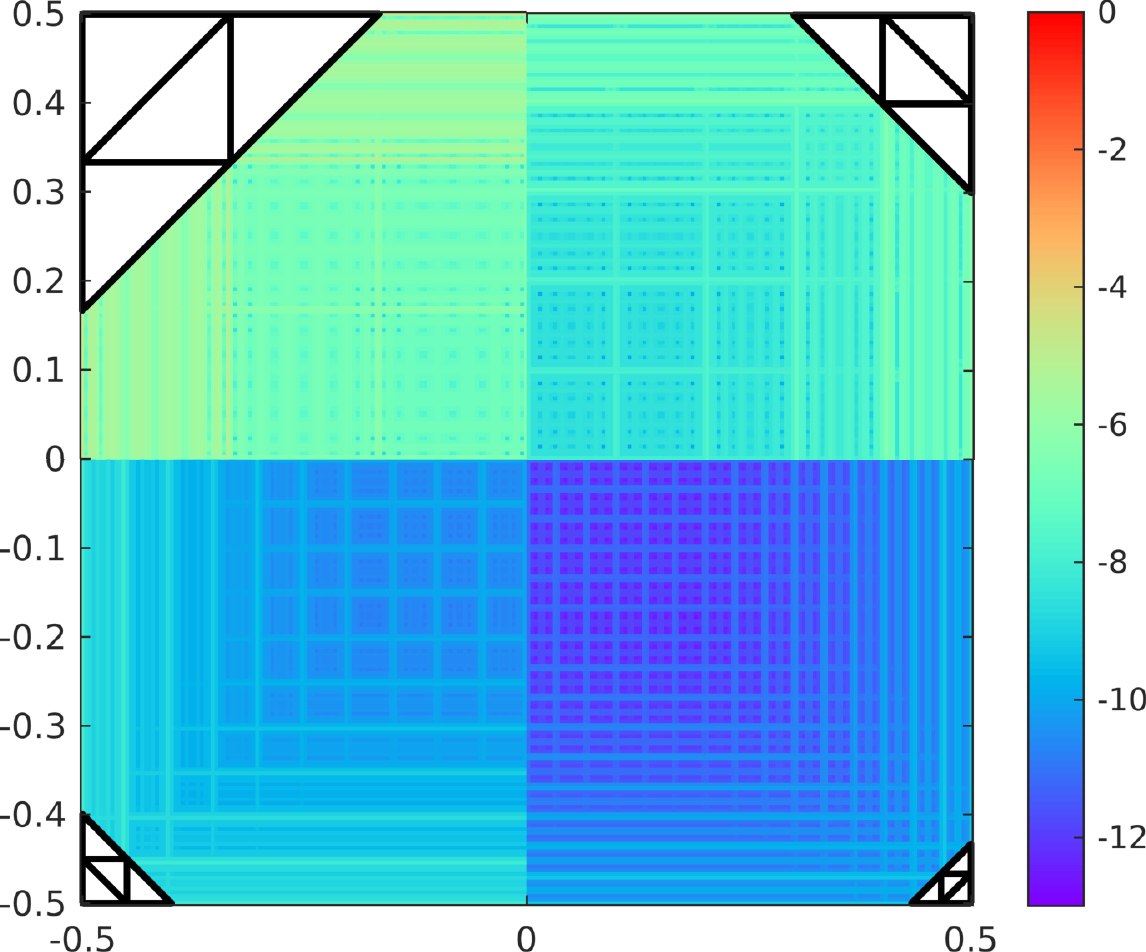}
        \includegraphics[height=0.28\textwidth]{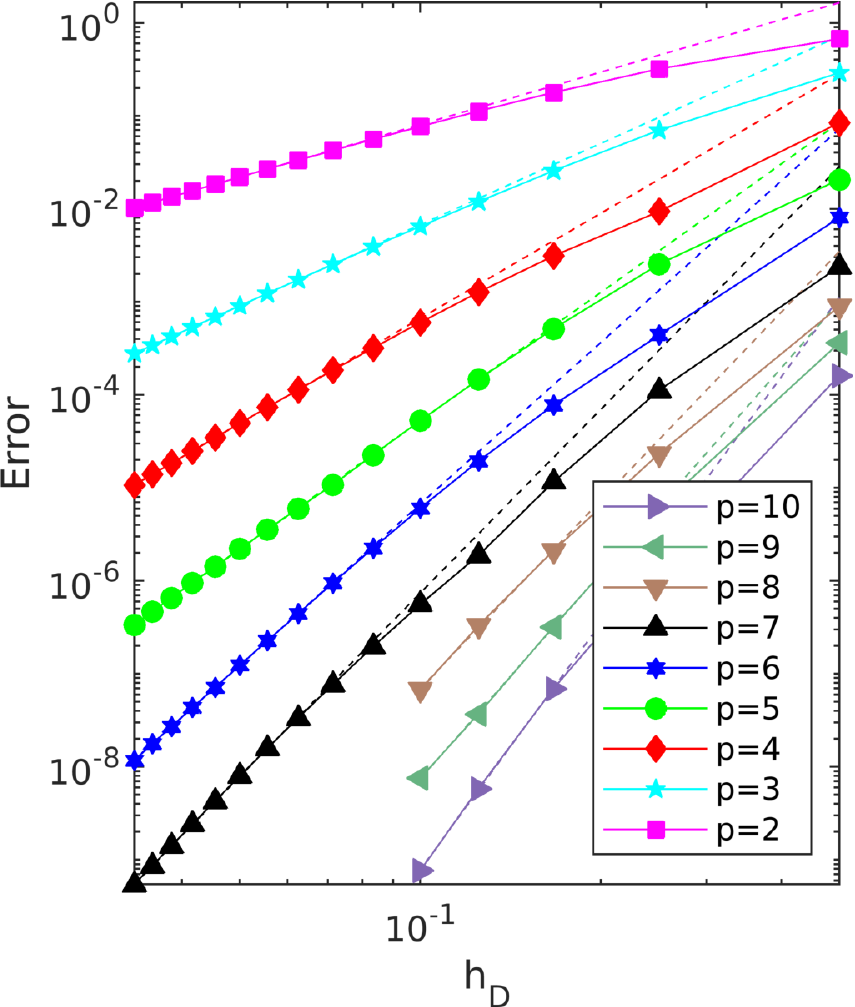}
        \caption{ {\em Left:} Two examples of the function pairs $(z, \mu)$ constructed using the procedure outlined in Sec \ref{subsec:numer_exp1}. In each case, the $z-$component is plotted with color scaled by the density $\mu$. {\em Middle:} The density approximation scheme \eqref{eq:q_approx} is applied to each of these function pairs and the max of the relative $L^{\infty}$ error is plotted here. The domain in $(x,y)$ is subdivided in the manner shown here, $h_D=1/6,\, 1/10,\, 1/20,\, 1/30$. The four skeletons on the corners illustrate refinements of the triangular grids at four corresponding quadrants. {\em Right:} Convergence plot of approximation error using $p=2,\cdots,7$. Dashed lines of corresponding colors are plots of expected error. For $p>7$, the generalized higher-order approximation scheme in  \eqref{eq:q_approx_higher} was employed.}\label{fig:numer_exp1}
    \end{figure}
	
\subsection{Laplace DLP evaluation test} 
	  Here, we consider the double-layer solution \emph{ansatz} \eqref{eq:introLapDLP} for the exterior Laplace equation. We begin with a cruller surface from Ref.  \cite{barnett2019high} which is parametrized by an infinitely differentiable, double $2\pi$-periodic function $\boldsymbol{r}(\theta,\phi)\,:\,[0,2\pi)^2\rightarrow\mathbb{R}^3$, followed by a cushion surface from Ref. \cite{perez2019harmonic}, which has a global parametrization in spherical coordinate $\boldsymbol{r}(\theta,\phi)\,:\,[-\pi/2,\pi/2]\times[0,2\pi)\rightarrow\mathbb{R}^3$. Lastly, to demonstrate the applicability of our approach to arbitrary parameterizations, we consider an input mesh without any analytic description of the geometry. 
	  
	  \emph{Example 1: ``Cruller" geometry.}
	    This example demonstrates the handling of complex geometries parameterized by doubly $2\pi$-periodic functions. We consider a smooth, warped torus surface $\mathcal{M}$ in Cartesian coordinates, given by
		\begin{equation}\label{eq:curller}
			\begin{aligned}
				\boldsymbol{r} = \boldsymbol{r}(\theta,\phi) = \left( \left(a+f\left(\theta,\phi\right)\cos\left(\theta\right)\right)\cos\left(\phi\right),\ \left(a+f\left(\theta,\phi\right)\cos\left(\theta\right)\right)\sin\left(\phi\right),\  f\left(\theta,\phi\right)\sin\left(\theta\right) \right)
			\end{aligned}
		\end{equation}
		where $f(\theta,\phi)=b+w_c\cos(w_n\phi+w_m\theta)$, $a=1$ and $b=1/2$. We test the close evaluation scheme using mean curvature $H$ as density:
		\begin{equation}
			\begin{aligned}
				\mu(\theta,\phi) = H(\theta, \phi) = \frac{1}{2} \left(E N -2FM+GL\right)/ \left(EG-F^2\right)
			\end{aligned}
		\label{eq:curvature}\end{equation}
		where the fundamental forms are given by $E = \boldsymbol{r}_{\phi}\cdot \boldsymbol{r}_{\phi}$, $F = \boldsymbol{r}_{\phi}\cdot \boldsymbol{r}_{\theta}$, $G = \boldsymbol{r}_{\theta}\cdot \boldsymbol{r}_{\theta}$, $L = \boldsymbol{r}_{\phi\phi}\cdot \boldsymbol{n}$, $M = \boldsymbol{r}_{\phi\theta}\cdot \boldsymbol{n}$, $N = \boldsymbol{r}_{\theta\theta}\cdot \boldsymbol{n}$, and the normal $\boldsymbol{n} = \left(\boldsymbol{r}_{\phi}\times \boldsymbol{r}_{\theta}\right)/|\boldsymbol{r}_{\phi}\times \boldsymbol{r}_{\theta}|$.
		
		We discretize $\mathcal{M}$ uniformly with $n_{\theta}$-by-$n_{\phi}$ rectangular patches and conduct a self-convergence study with reference values computed using the close evaluation scheme with $n_{\theta} = 108, n_{\phi} = 144$. The results are shown in Table \ref{tab:torus_eval_error} and Fig. \ref{fig:torus_eval}. The errors are measured at targets located on the $\phi = \pi/2$ plane (as shown in Fig. \ref{fig:torus_eval}, left). We note that the expected order of convergence is observed asymptotically.
		
		\begin{figure}[h!]
            \centering
            \includegraphics[height=0.28\textwidth]{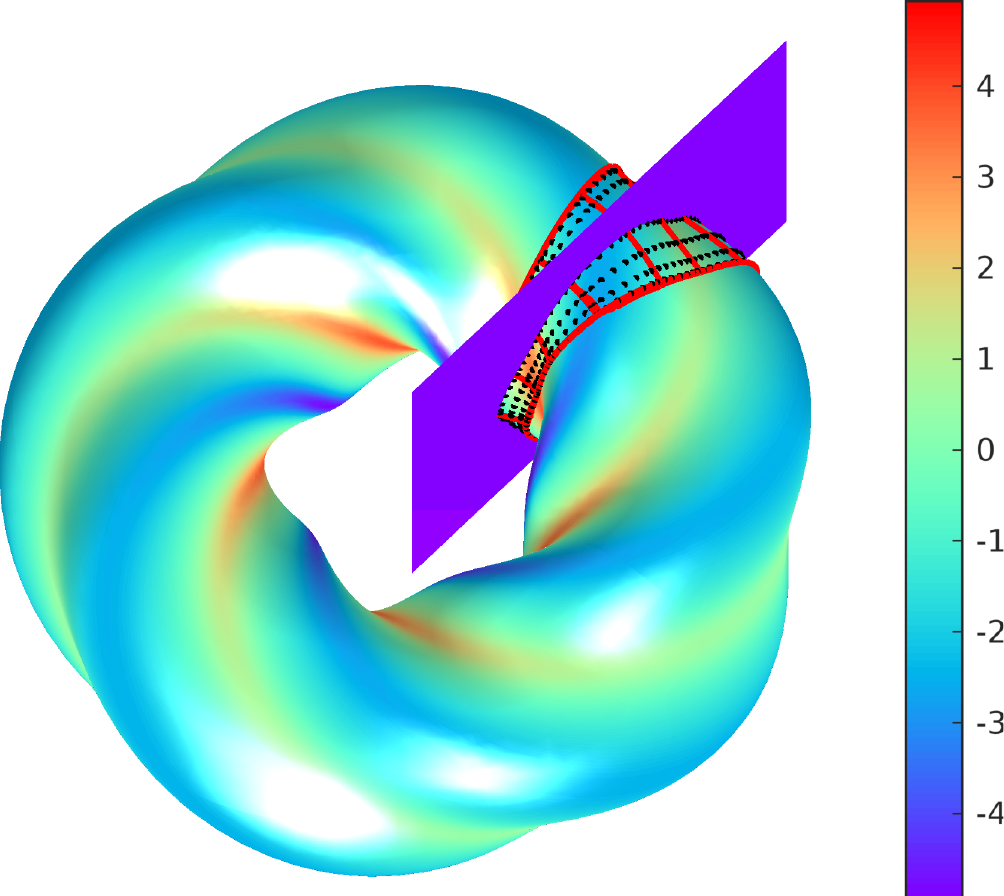}
            \includegraphics[height=0.28\textwidth]{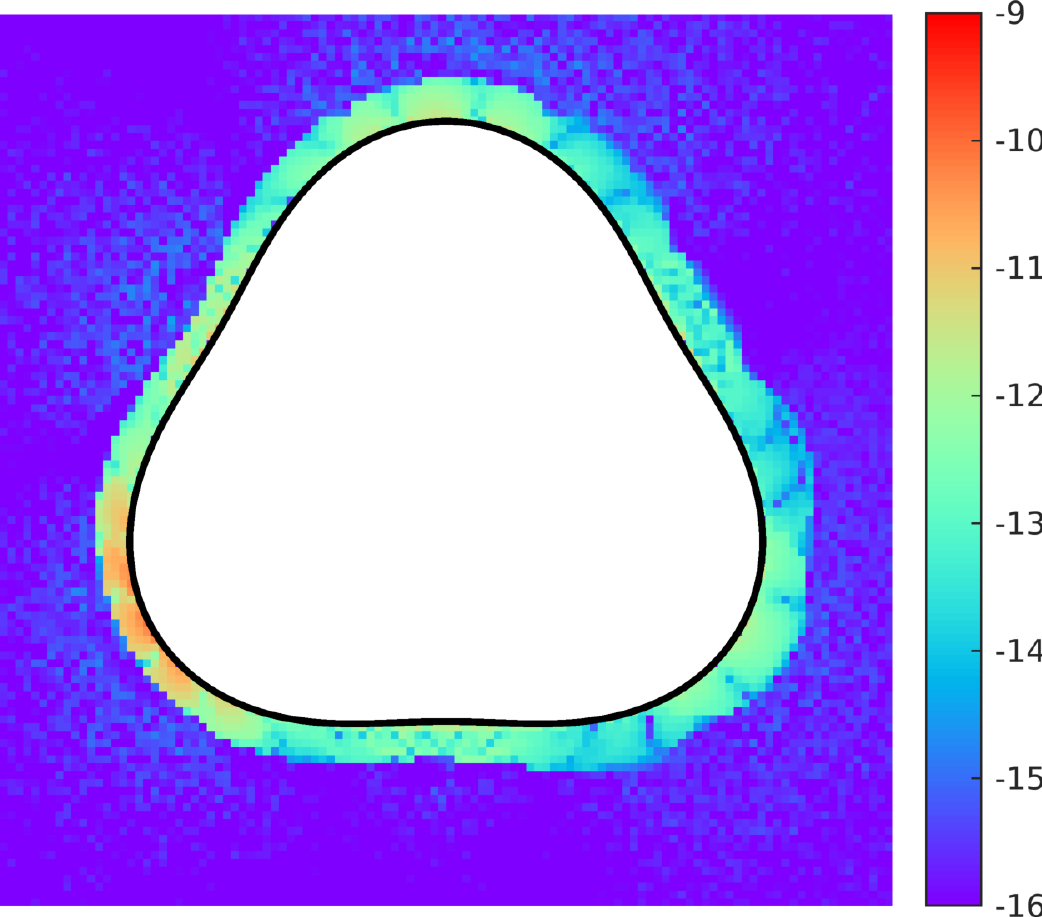}
            \includegraphics[height=0.28\textwidth]{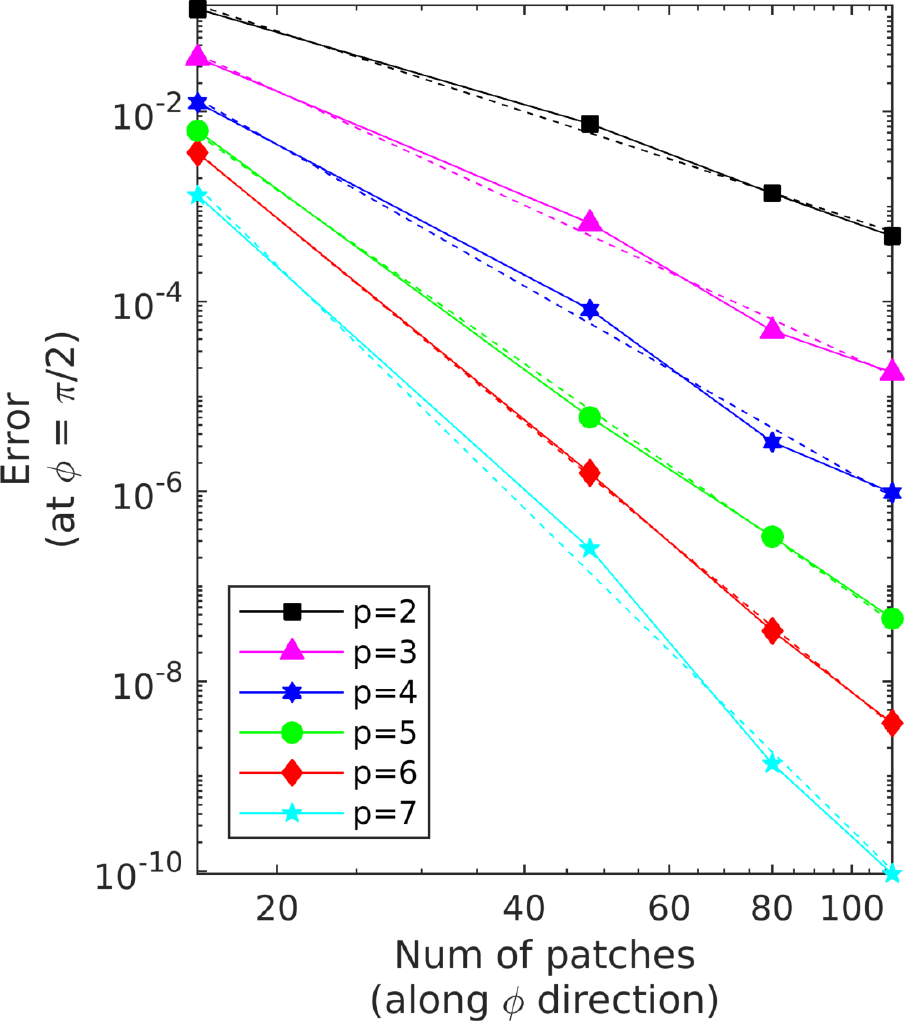}
            \caption{ Laplace DLP close evaluation scheme on a smooth, warped torus surface, using mean curvature as prescribed density. Left: Surface $\mathcal{M}$, with $w_c=0.065$, $w_m=3$ and $w_n=5$, showing panel divisions (red lines) intersecting $YZ$-plane and Nystr\"om nodes (black). The color indicates the magnitude of mean curvature. Middle: Cross-section view of the $\log_{10}$ relative error in the exterior of $\mathcal{M}$, in the $YZ$-plane ($\phi=\pi/2$), with $84\times 112$ patches. Right: Rate of convergence of the relative errors across the same shown slice with respect to number of panels along torodial direction, for $p=2,\cdots,7$.}\label{fig:torus_eval}
        \end{figure}
		
		\begin{table}[t]
  		\centering
			\begin{tabular}{|c||p{1.40cm}|p{0.65cm}||p{1.40cm}|p{0.65cm}||p{1.40cm}|p{0.65cm}||p{1.40cm}|p{0.65cm}||p{1.40cm}|p{0.65cm}|}
                	\hline
                	&\multicolumn{2}{c||}{$p=3$}& \multicolumn{2}{c||}{$4$} & \multicolumn{2}{c||}{$5$}& \multicolumn{2}{c||}{$6$} & \multicolumn{2}{c|}{$7$} \\ \hline
                	$n_{\theta}\times n_{\phi}$ & $\text{max } E_{rel}$ & $\hat{p}$ & $\text{max } E_{rel}$ & $\hat{p}$ & $\text{max } E_{rel}$ & $\hat{p}$ & $\text{max } E_{rel}$ & $\hat{p}$ & $\text{max } E_{rel}$ & $\hat{p}$ \\ \hline
                	12$\times$ 16   & 3.70e-02  &       & 1.25e-02  &       & 6.30e-03  &       & 3.68e-03  &       & 1.32e-03  &   \\
                	36$\times$ 48   & 6.72e-04  &  3.65 & 8.27e-05  &  4.57 & 6.05e-06  &  6.32 & 1.56e-06  &  7.07 & 2.51e-07  &  7.80\\
                	60$\times$ 80   & 4.92e-05  &  5.12 & 3.32e-06  &  6.30 & 3.36e-07  &  5.66 & 3.40e-08  &  7.50 & 1.35e-09  &  10.2\\
                	84$\times$ 112  & 1.79e-05  &  3.00 & 9.67e-07  &  3.67 & 4.59e-08  &  5.91 & 3.61e-09  &  6.66 & 9.37e-11  &  7.94\\
                	\hline
            \end{tabular}
			\caption{ Laplace DLP close evaluation scheme using mean curvature as prescribed density in the exterior of a smooth, warped torus surface parametrized by $(\theta,\phi) \in [0,2\pi)^2$, with $w_c=0.065$, $w_m=3$ and $w_n=5$. A cross-section on the $YZ$-plane ($\phi = \pi/2$) is chosen to study the convergence. We report both the maximum relative error and the observed convergence rate ($\hat{p}$) across the same slice as the number of panels are increased. }\label{tab:torus_eval_error}
		\end{table}	
		
	  \emph{Example 2: ``Cushion" geometry.}	
		In this example, we consider the cushion surface $\mathcal{M}$ from Ref.  \cite{perez2019harmonic} defined by 
		\begin{equation}
			\begin{aligned}
				\mathcal{M} = \boldsymbol{r}(\theta,\phi) = \left( f\left(\theta,\phi\right)\cos\left(\theta\right)\cos\left(\phi\right),\  f\left(\theta,\phi\right)\sin\left(\theta\right)\cos\left(\phi\right),\  f\left(\theta,\phi\right)\sin\left(\phi\right) \right)
			\end{aligned} 
		\label{cushion}\end{equation} 
		where $f(\theta,\phi)=\left(4/5+1/2\left(\cos\left(2\theta\right)-1\right)\left(\cos\left(4\phi\right)-1\right)\right)^{1/2}$. The density $\mu$ is set to the mean curvature \eqref{eq:curvature}. As in \cite{perez2019harmonic}, we discretize $\mathcal{M}$ using a set of non-overlapping patches and tensor product grids in each patch (Fig. \ref{fig:cushion_eval}, left). The DLP is evaluated on two planes, intersecting the surface at $\phi = 3\pi/16$ and $\phi = 27\pi/16$, as shown in \ref{fig:cushion_eval}.  The relative errors are plotted as the number of patches are increased for the case of $p =7$. While the experimental setup is slightly different from that in \cite{perez2019harmonic}, our goal is to showcase that several more digits of accuracy can be obtained using higher order close evaluation schemes. For example, from Fig. \ref{fig:cushion_eval}, we observe that around 10-digits of accuracy (or better) can be achieved at targets that are arbitrarily close to the surface on both the planes. 
		\begin{figure}[h!]
            \centering
            \includegraphics[height=0.28\textwidth]{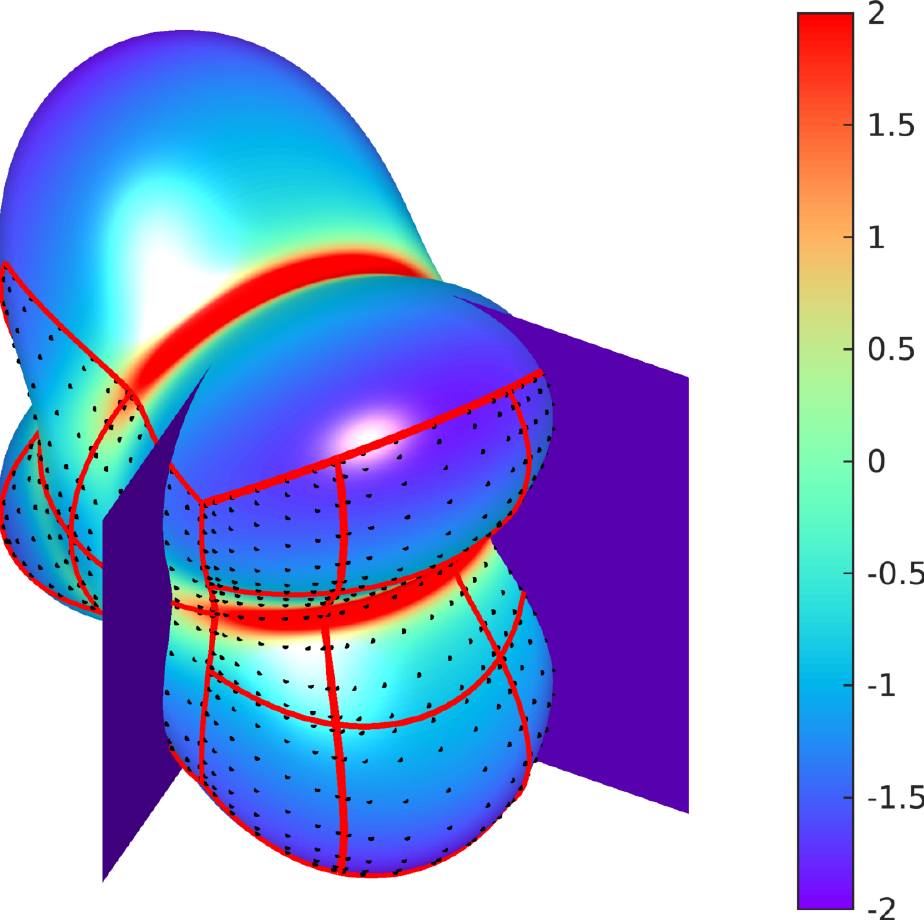}
            \includegraphics[height=0.28\textwidth]{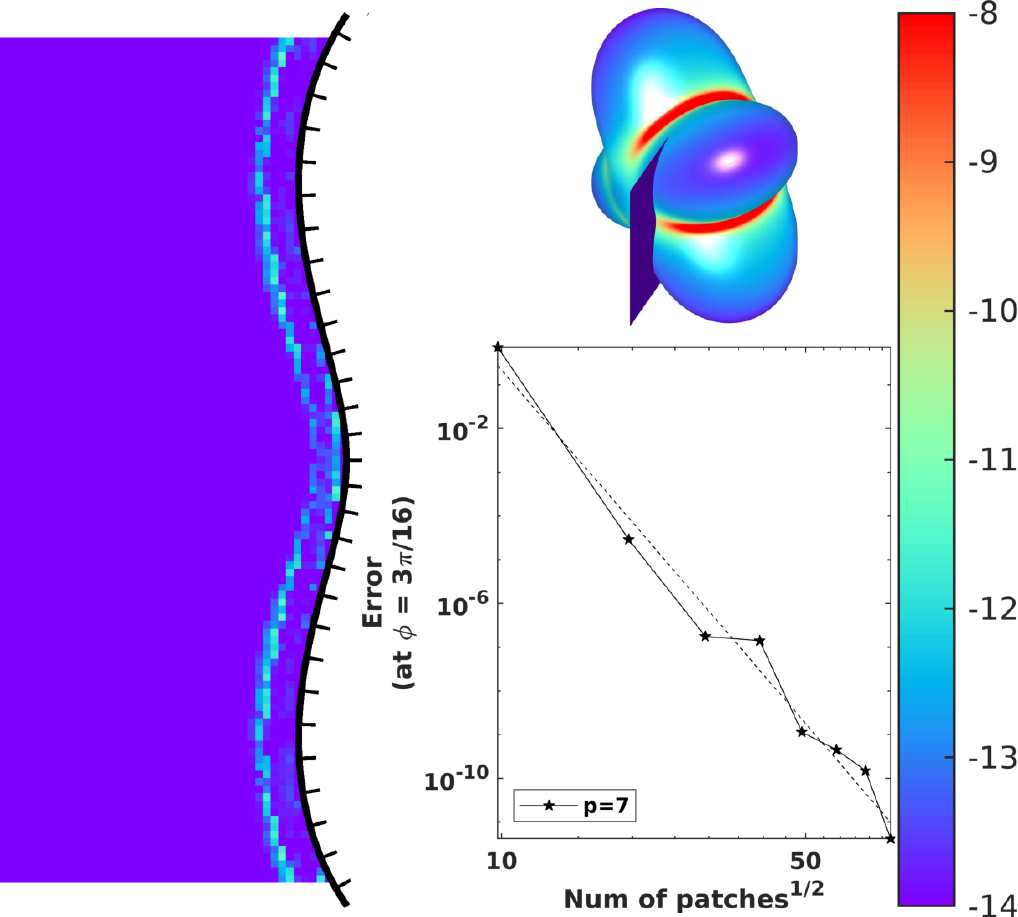}
            \includegraphics[height=0.28\textwidth]{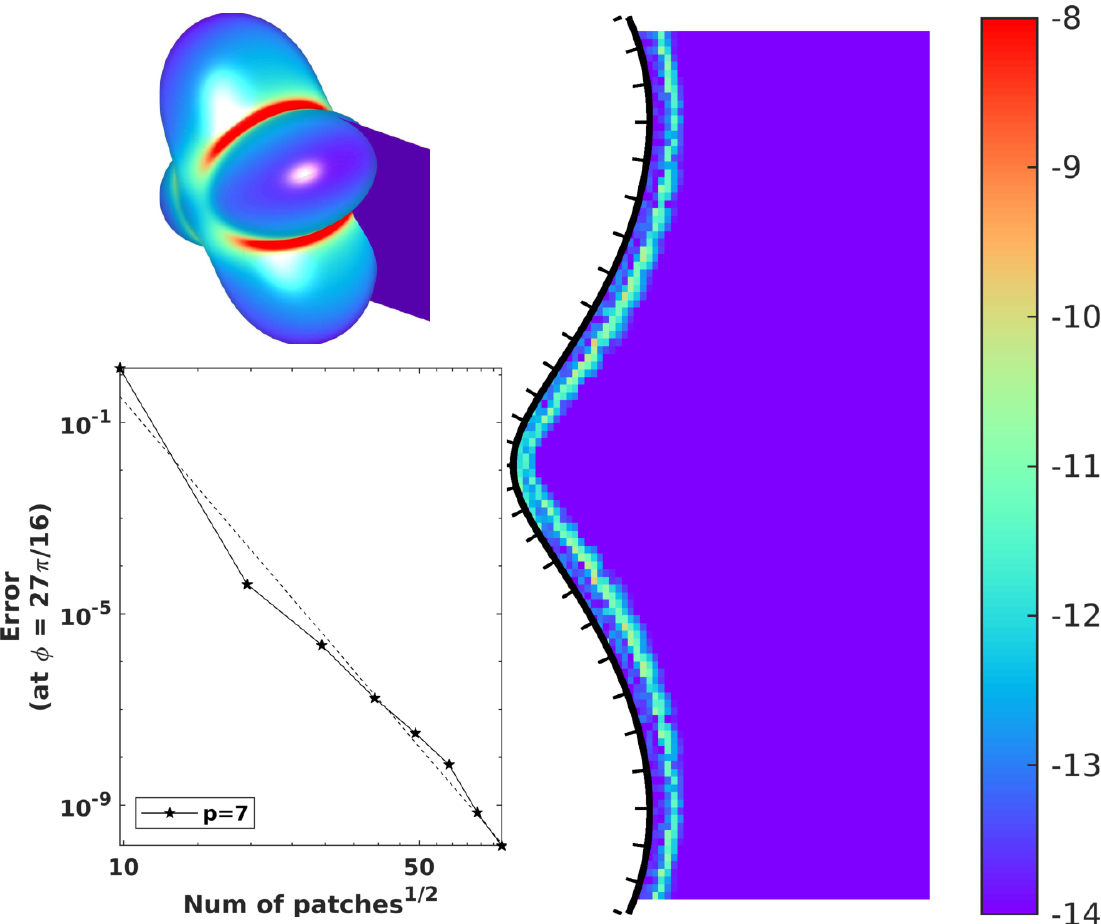}
            \caption{ Laplace DLP close evaluation on a cushion-shaped geometry. Left: Illustration of the surface discretization with non-overlapping patches. The magnitude of mean curvature is indicated by the color. The solution is evaluated on the shown slices. Middle: Cross-section view on the plane $\phi = 3\pi/16$ of the $\log_{10}$ relative error in the exterior of the cushion. The inset plots the relative error corresponding to $p=7$ as a function of the number of patches. Max relative error is $4.1314\times 10^{-12}$ with a total number of $6144$ ($32\times 32 \times 6$) patches. Current surface discretization is shown by the ticks ($'\vert'$) along the surface. Right: The $\log_{10}$ relative error in the exterior of the cushion on the plane $\phi = 27\pi/16$ with a total number of $6144$ patches. Max relative error is $1.4157\times 10^{-10}$.}\label{fig:cushion_eval}
        \end{figure}
        
      \emph{Example 3: ``Bunny" geometry.}
        Finally, we showcase that when a smooth surface is given without $(\theta,\phi)$ parametrization, a local correction could still be implemented to higher order based on a local set of control points and a high order polynomial interpolation. The bunny geometry is taken from a standard mesh library. We use an interactive high-quality quad remeshing tool developed in \cite{takayama2013sketch} to return a collection of patches and nodes within. We then construct a $6^{\text{th}}$ order polynomial approximation to each patch, and generate Gauss-Legendre quadrature nodes for naive evaluation and boundaries of each patch for applying our close evaluation scheme. We use a randomly chosen density function $\mu(x,y,z)=e^{xy}-1+x+\sin(x^4+1/2y^3)+y-1/2y^2+1/5y^6+z$ on the bunny surface. The results on one patch are shown in Fig. \ref{fig:bunny} to demonstrate the performance.
		
	\subsection{Laplace BVP test}
        For this numerical experiment, we solve Laplace interior boundary value problems (BVP) inside two cruller \eqref{eq:curller} domains, one with higher curvature compared to the other, as shown in Figs. \ref{fig:torus_bvp1} and \ref{fig:torus_bvp2}. The boundary data is generated from a superposition of randomly distributed point sources located exterior to the domain, i.e., we evaluate the function $g(\vv{r})=\sum_{j=1}^{N_s} G(\vv{r}-\vv{r}_j) h_j$, where $\vv{r}\in \mathcal{M}$, the sources $\vv{r}_j$ are located in the exterior and the source strengths $h_j$ are set to some random values. Starting from this boundary data, the BIE \eqref{eq:laplace} can be solved for unknown density function $\mu$. We use the quadrature scheme and BIE solver developed in \cite{barnett2019high} for this purpose. The resulting density $\mu$ for both geometries are plotted in Fig. \ref{fig:torus_bvp1} (left) and Fig. \ref{fig:torus_bvp2} (left). The numerical solution of the BVP is evaluated on the $\phi=\pi/8$ plane using our close evaluation scheme with $p=7$. It is then compared against the exact solution $u_{\text{exact}}(\vv{r}')=\sum_{j=1}^{N_s} G(\vv{r}',\vv{r}_j)h_j$, where $\vv{r}'\in\mathbb{D}$. We can make the following observations from Figs. \ref{fig:torus_bvp1} and \ref{fig:torus_bvp2}: (i) the accuracy is uniform throughout the interior (that is, no degradation at targets close to the boundary), (ii) similar to the exterior problem case, order of convergence is consistent with the basis function space used, and (iii) level of accuracy achieved is consistent with the complexity of the geometry.  
      
		\begin{figure}[h!]
    		\centering
    		\includegraphics[height=0.28\textwidth]{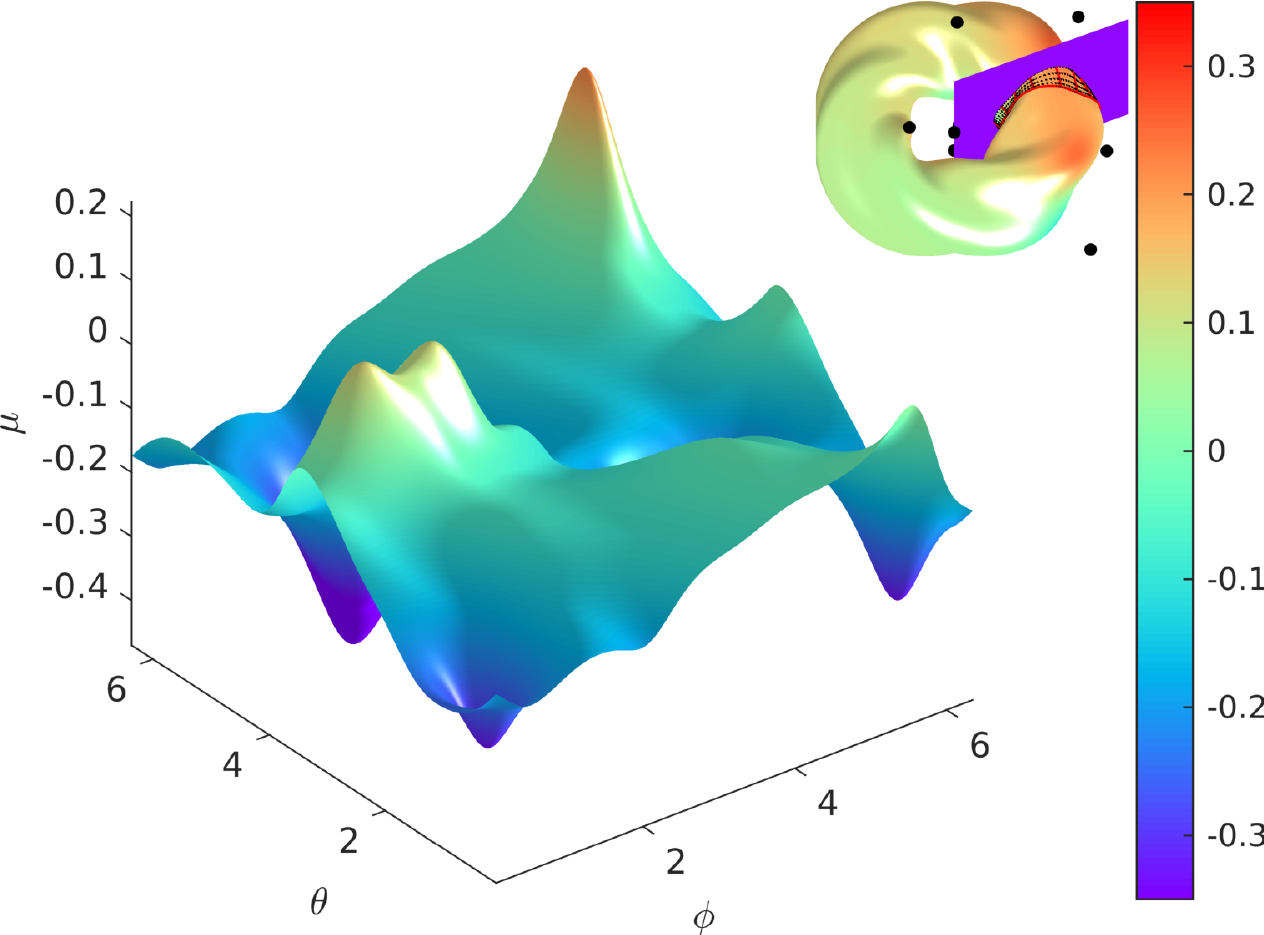}
    		\includegraphics[height=0.28\textwidth]{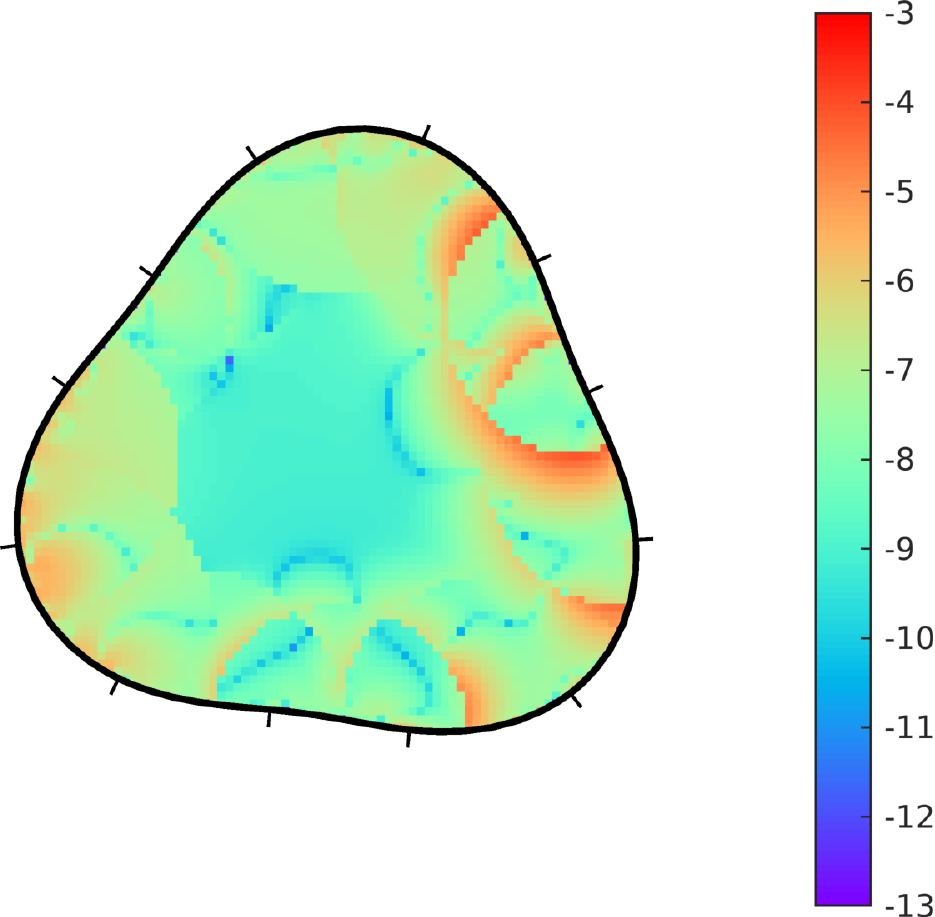}
			\includegraphics[height=0.28\textwidth]{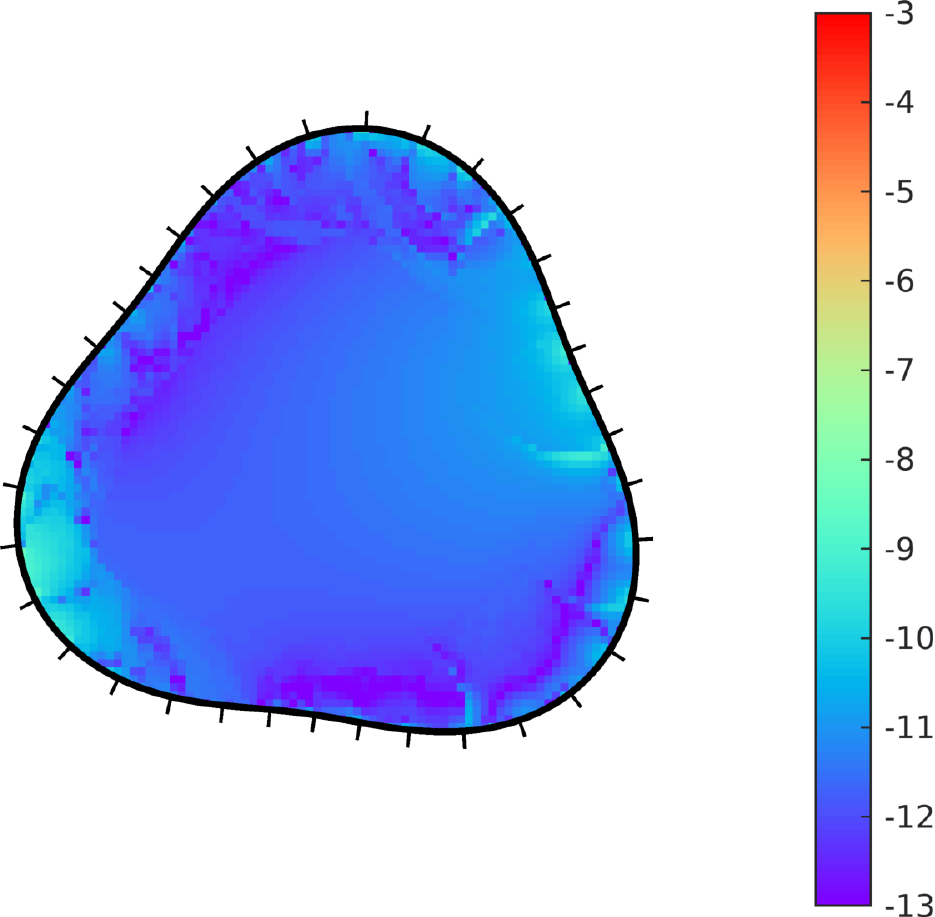}
			\caption{ Solution of Laplace BVP in the interior of a torus, using an indirect DLP formulation \eqref{eq:laplace}. Left: Density function plotted as a function of torodial and poloidal directions. The inset at the upper right corner shows the geometry whose surface color indicates the Dirichlet data due to a few randomly placed sources (black dots) in the exterior. The solution is evaluated on the shown slice. Here the shape parameters in \eqref{eq:curller} were set to $w_c=0.065$, $w_m=3$ and $w_n=5$. Middle: Cross-section view of the $\log_{10}$ relative error on the plane $\phi=\pi/8$. Max relative error is $7.1761\times 10^{-5}$ with $12\times 16$ panels. Right: The $\log_{10}$ relative error on the same shown slice with $36\times 48$ panels. Max relative error is $3.7405\times 10^{-9}$.}\label{fig:torus_bvp1}
        \end{figure}
		
		\begin{figure}[h!]
    		\centering
    		\includegraphics[height=0.28\textwidth]{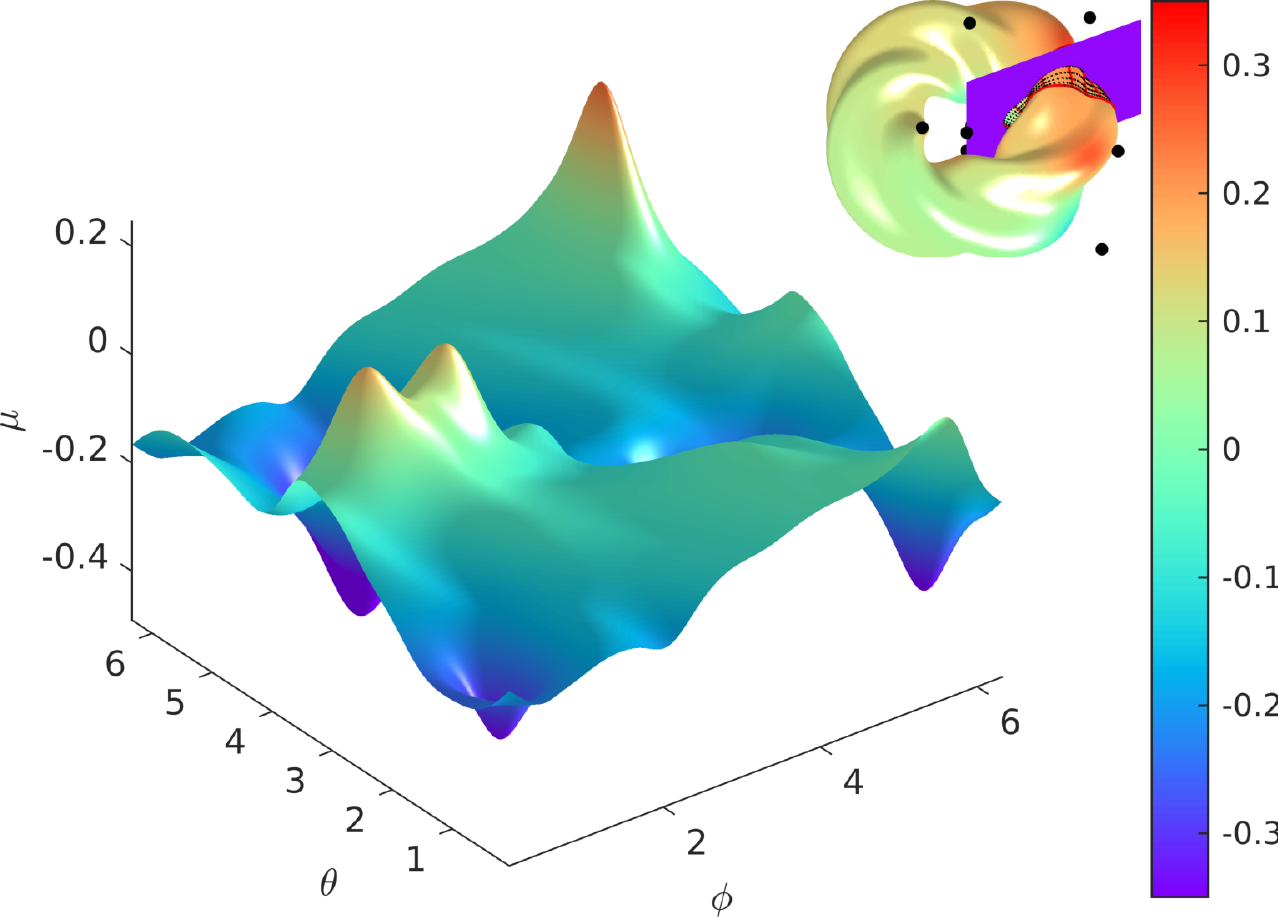}
    		\includegraphics[height=0.28\textwidth]{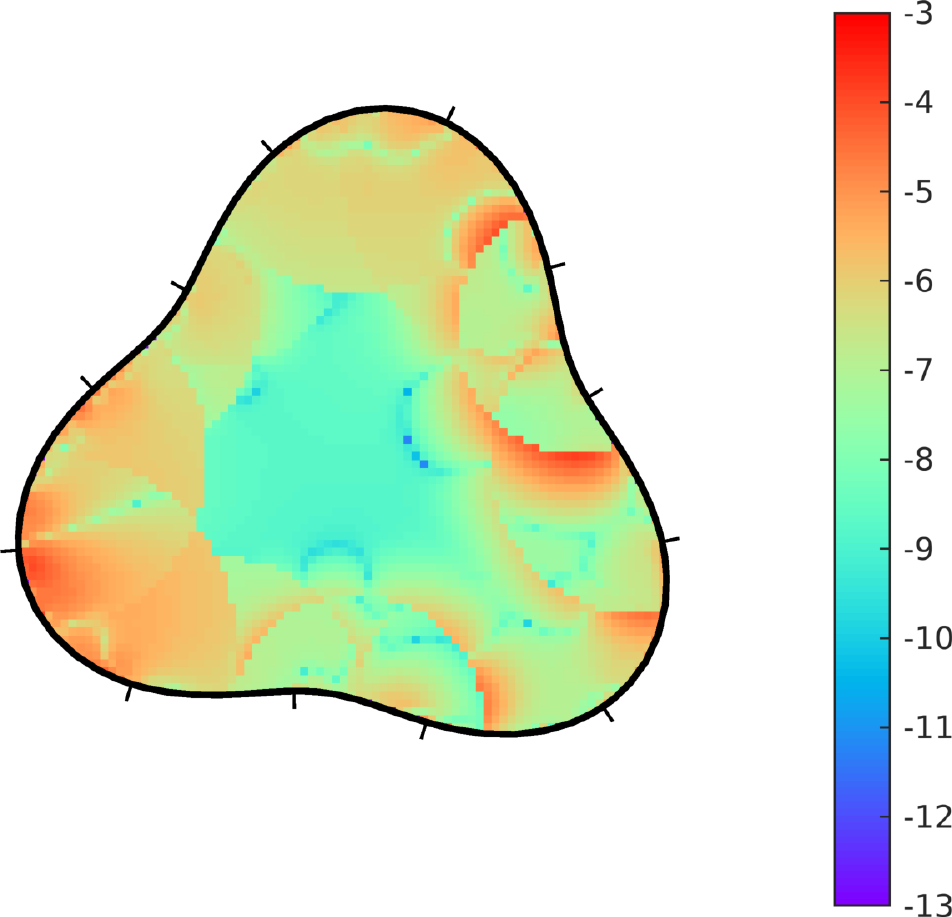}
    		\includegraphics[height=0.28\textwidth]{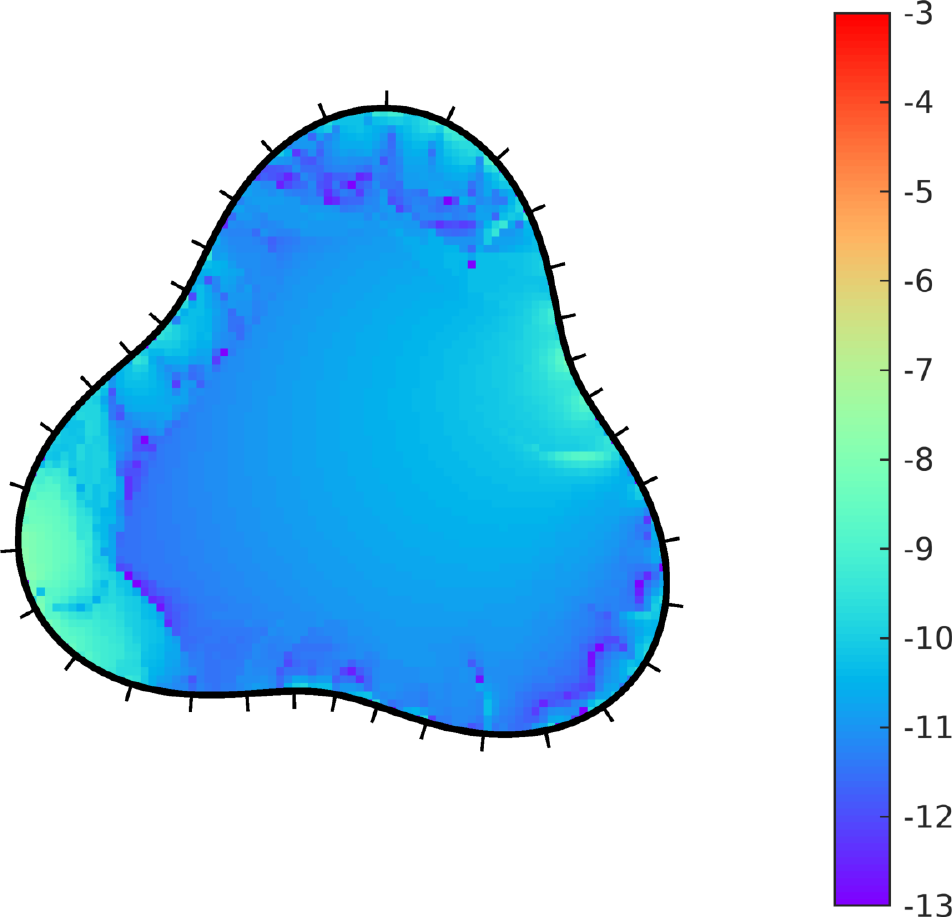}
    		\caption{ Same setup as in Fig. \ref{fig:torus_bvp1} but with shape parameters $w_c=0.1$, $w_m=3$ and $w_n=5$ (higher curvature). The max relative error is $1.6098\times 10^{-4}$ with $12\times 16$ panels (middle) and  with $36\times 48$ panels, it is $1.6667\times 10^{-8}$ (right). }
	        \label{fig:torus_bvp2}
        \end{figure}

\section{Conclusions}
\label{Concl}
In summary, we presented a high-order technique for evaluating nearly singular integrals for Laplace layer potentials in three dimensions and demonstrated its efficacy on a range of test problems. This scheme has modest requirements: it can work on any user-supplied surface mesh directly to solve the close evaluation problem up to the level of accuracy commensurate with that of the given data. Moreover, to some extent, this scheme is {\em dimension-agnostic}. It is intriguing to note that if we carry out the same steps for two-dimensional DLPs, we will likely recover the scheme of Helsing-Ojala \cite{helsing2008evaluation}. There, nearly singular integrals were computed using a quadrature scheme that employs piecewise complex monomial approximation on panels, Cauchy’s theorem and recurrence relations. In our case, we can use harmonic polynomials in two dimensions---which are closely related to complex monomials---for density approximation and the DLP can be transformed from a $1$-form line integral to $0$-form antiderivative evaluation (i.e., using \eqref{eq:alpha} and \eqref{eq:omega}). 

We plan to extend our work on several fronts. Our immediate next step is to integrate the close evaluation routine with an open-source FMM package (e.g., \cite{greengard2021fast}) and test its performance on large-scale examples. Another natural direction is to consider various other elliptic PDE kernels including Helmholtz, Stokes and Navier kernels. As indicated earlier, such a task is non-trivial since the density approximation likely needs to be modified and new recurrences for 1-forms need to be derived. 
%applications
Lastly, accurate three-dimensional close evaluation schemes open up possibilities to investigate physical phenomena that are otherwise hard to simulate including chain formation and chaotic behavior in vesicle electrohydrodynamics \cite{veerapaneni2016integral, wu2019electrohydrodynamics}, flows through complex geometries \cite{marple2016fast, wu2019solution} and self-assembly of active particles \cite{yan2020scalable, kohl2021fast}. We plan to generalize these previous works to large-scale three-dimensional flows with arbitrary particle shapes. 

\section{Acknowledgements} 
We thank Alex Barnett, Charles Epstein, Leslie Greengard and Manas Rachh for many useful discussions pertaining to this work. We acknowledge support from NSF under grants DMS-1719834, DMS-1454010 and DMS-2012424, and the Mcubed program at the University of Michigan. The work of SV was also supported by the Flatiron Institute, a division of the Simons Foundation.

\appendix
\section{Proof of Lemma~\ref{lemma:potential}}\label{appx:verify_poin}
	\begin{proof}
        Denoting $P(\vv{r}) = \int_0^1\left(tzg_2(t\vv{r})-tyg_3(t\vv{r})\right)dt$, $Q(\vv{r}) = \int_0^1\left(txg_3(t\vv{r})-tzg_1(t\vv{r})\right)dt$, and $R(\vv{r}) = \int_0^1\left(tyg_1(t\vv{r})-txg_2(t\vv{r})\right)dt$, the exterior derivative of \eqref{eq:poincare_r3}, $\omega = P(\vv{r}) dx+Q(\vv{r})dy+R(\vv{r})dz$, is 
        \begin{equation}
                \begin{aligned}
                        d\omega = \left(\frac{\partial R}{\partial y}-\frac{\partial Q}{\partial z}\right)dy\wedge dz + \left(\frac{\partial P}{\partial z}-\frac{\partial R}{\partial x}\right)dz\wedge dx + \left(\frac{\partial Q}{\partial x}-\frac{\partial P}{\partial y}\right)dx\wedge dy.
                \end{aligned}
        \end{equation}
        The first term, $\frac{\partial R}{\partial y}-\frac{\partial Q}{\partial z}$, can be expanded as 
        \[
                \begin{aligned}
                         &\frac{\partial }{\partial y}\left(\int_0^1\left(tyg_1(t\vv{r})-txg_2(t\vv{r})\right)dt\right)-\frac{\partial }{\partial z}\left(\int_0^1\left(txg_3(t\vv{r})-tzg_1(t\vv{r})\right)dt\right)\\
                         =&\int_0^1 \left(tg_1(t\vv{r})+t^2yg_{1,y}(t\vv{r})-t^2xg_{2,y}(t\vv{r})\right)dt - \int_0^1 \left(t^2xg_{3,z}(t\vv{r})-tg_1(t\vv{r})-t^2zg_{1,z}(t\vv{r})\right)dt\\
                         =&\int_0^1\left(2tg_1(t\vv{r})+t^2yg_{1,y}(t\vv{r})+t^2zg_{1,z}(t\vv{r})-t^2x\left(g_{2,y}(t\vv{r})+g_{3,z}(t\vv{r})\right)\right)dt\\
                         =&\int_0^1\left(2tg_1(t\vv{r})+t^2yg_{1,y}(t\vv{r})+t^2zg_{1,z}(t\vv{r})+t^2xg_{1,x}(t\vv{r})\right)dt \qquad (\text{since}\quad \nabla \cdot \vv{g} = 0)\\
                         =&\int_0^1\frac{d}{dt}(t^2g_1(t\vv{r}))dt\, = g_1(\vv{r}).
                \end{aligned}
        \]
         Treating the other two terms, $\frac{\partial P}{\partial z}-\frac{\partial R}{\partial x}$ and $\frac{\partial Q}{\partial x}-\frac{\partial P}{\partial y}$, similarly, we get the result $d\omega = g_1dy\wedge dz + g_2dz\wedge dx + g_3dx\wedge dy$.
	\end{proof}       

\section{Recurrence relations for evaluating moments}\label{appx:LMN}
Here, we present recurrence relations for evaluating moments of high-order monic polynomials w.r.t certain kernels that arise in $1$-form construction. Together with $M_k$ that is required in \eqref{eq:1-form}, we need $L_k$ and $N_k$, defined below, for evaluating other Laplace layer potentials:  
    \begin{equation}
        \begin{aligned}
            L_k(\vv{r}',\vv{r}) = \int_0^1 \frac{t^k}{|t\vv{r}-\vv{r}' |^{5}}dt,\quad M_k(\vv{r}',\vv{r}) = \int_0^1 \frac{t^k}{|t\vv{r}-\vv{r}' |^{3}}dt, \quad  N_k(\vv{r}',\vv{r}) =\int_0^1 \frac{t^k}{|t\vv{r}-\vv{r}' |}dt
        \end{aligned}
     \end{equation}
Using integration by parts and after some algebra, we can arrive at the following recurrences for evaluating the above moments:
     \begin{equation}
        \begin{cases}
                L_k = \frac{2 \vv{r}' \cdot \vv{r} }{| \vv{r} |^2}L_{k-1}-\frac{| \vv{r}' |^2}{| \vv{r} |^2}L_{k-2}+\frac{1}{| \vv{r} |^2}M_{k-2}\\
                M_k = \frac{ \vv{r}' \cdot \vv{r} }{| \vv{r} |^2}M_{k-1}+  \frac{k-1}{| \vv{r} |^2}N_{k-2}- \left. \frac{1}{| \vv{r} |^2}\frac{t^{k-1}}{|t \vv{r} - \vv{r}' |} \right\rvert_0^1\\
                N_k = \frac{2k-1}{k}\frac{\vv{r}' \cdot \vv{r}}{| \vv{r} |^2} N_{k-1} -  \frac{k-1}{k} \frac{| \vv{r}' |^2}{| \vv{r} |^2} N_{k-2} + \left. \frac{1}{| \vv{r} |^2}\frac{t^{k-1}|t \vv{r} - \vv{r}' |}{k}\right\rvert_0^1
        \end{cases}
     \end{equation}
The base conditions for these recurrences can also easily be derived as
\renewcommand{\arraystretch}{2}
        \begin{equation} \label{eq:base}
        \setlength{\jot}{5pt}
                \begin{aligned}
                        N_0 & = \frac{1}{| \vv{r} |}\left( \log\left(| \vv{r} || \vv{r} - \vv{r}' |+| \vv{r} |^2-\left( \vv{r}' \cdot \vv{r} \right)\right) -\log\left(| \vv{r}' || \vv{r} |-\left( \vv{r}' \cdot \vv{r} \right)\right)\right), \\
                        N_1 & = \frac{1}{| \vv{r} |^2}\left( | \vv{r} - \vv{r}' |-| \vv{r}' |\right) + \frac{\left( \vv{r}' \cdot \vv{r} \right)N_0}{| \vv{r} |^2},  \\
                        M_0 &= \frac{| \vv{r} |}{| \vv{r} |^2| \vv{r}' |^2-\left( \vv{r}' \cdot \vv{r} \right)^2}\left(\frac{| \vv{r} |^2-\left( \vv{r}' \cdot \vv{r} \right)}{| \vv{r} || \vv{r} - \vv{r}' |}+\frac{ \vv{r}' \cdot \vv{r} }{| \vv{r} || \vv{r}' |}\right),\\
                        M_1 &= \frac{1}{\left( \vv{r}' \cdot \vv{r} \right)^2-| \vv{r}' |^2| \vv{r} |^2}\left(\frac{- \vv{r}' \cdot \vv{r} +| \vv{r}' |^2}{| \vv{r} - \vv{r}' |}-| \vv{r}' |\right),\\
                        L_0 &= \frac{1}{\left( | \vv{r}' |^2| \vv{r} |^2-\left( \vv{r}' \cdot \vv{r} \right)^2\right)^2}\left(| \vv{r} |^2\frac{| \vv{r} |^2- \vv{r}' \cdot \vv{r} }{| \vv{r} - \vv{r}' |}-\frac{1}{3}\frac{\left(| \vv{r} |- \vv{r}' \cdot \vv{r} \right)^3}{| \vv{r} - \vv{r}' |^3}+| \vv{r} |^2\frac{ \vv{r}' \cdot \vv{r} }{| \vv{r}' |}-\frac{1}{3}\frac{\left( \vv{r}' \cdot \vv{r} \right)^3}{| \vv{r}' |^3}\right), \\
                        L_1 &= -\frac{1}{| \vv{r} |^2}\left(\frac{1}{3}\frac{1}{| \vv{r} - \vv{r}' |^3}-\frac{1}{3}\frac{1}{| \vv{r}' |^3}\right) + \frac{\left( \vv{r}' \cdot \vv{r} \right)}{| \vv{r} |^2}L_0.       
                \end{aligned}
        \end{equation}

\section{Second-order approximation scheme for Laplace double-layer potential}\label{appx:dlp_2nd_order}
Here, we provide further details on the steps outlined in Section \ref{sc:LapDeval} by considering the simpler $p = 2$ case and give explicit formulas for all the intermediate operators and functions. 
	\paragraph{Stage 1: Precomputation} 
	Recall that the first step is the change of coordinates wherein $\vv{r}^{(1,1)}$ becomes the origin and rest of the quadrature nodes are transformed accordingly. In the case of $p = 2$, there are eight unknowns, four elements each of the vectors $C^{(2,1)}$ and $C^{(2,2)}$ (defined in \eqref{eq:quaternion_xmatrix}). We can explicitly write the two vector equations obtained by applying \eqref{eq:quaternion_algo} as 
				\begin{equation}
					\begin{aligned}					 A[f^{(2,1)}](\tilde{\vv{r}}^{(2,1)})C^{(2,1)}+A[f^{(2,2)}](\tilde{\vv{r}}^{(2,1)})C^{(2,2)} & = U^{(2,1)}, \\
						  A[f^{(2,1)}](\tilde{\vv{r}}^{(2,2)})C^{(2,1)}+A[f^{(2,2)}](\tilde{\vv{r}}^{(2,2)})C^{(2,2)} & = U^{(2,2)}.
					\end{aligned}
				\end{equation}
				From \eqref{eq:quaternion_xmatrix}, we can expand the matrix operators as
				\begin{equation}
					\begin{aligned}
						A[f^{(2,1)}](\vv{r}) = \begin{pmatrix}
    						0 & -x & 0 & z \\
    						x & 0 & z & 0 \\
    						0 & -z & 0 & -x \\
    						-z & 0 & x & 0 
    						\end{pmatrix}\quad\text{and}\quad  A[f^{(2,2)}](\vv{r}) = \begin{pmatrix}
    						 0 & 0 & -y & z \\
    						 0 & 0 & z & y \\
    						 y & -z & 0 & 0\\
    						-z & -y & 0 & 0
    						\end{pmatrix}.
				\end{aligned}
				\end{equation}
	\paragraph{Stage 2: 2-to-1 form conversion and contour integration} The $1$-form $\omega$ for linear case can be carried out relatively easy. We only have two quaternionic $2$-forms (we omit \textasciitilde{} in $\tilde{\vv{r}}$):
				\begin{equation}
        			\begin{cases}
                		\alpha_0^{(2,1)} = &\frac{(y'-y)z}{|\vv{r}'-\vv{r}|^3}dy\wedge dz+\frac{-(z'-z)x-(x'-x)z}{|\vv{r}'-\vv{r}|^3}dz\wedge dx+\frac{(y'-y)x }{|\vv{r}'-\vv{r}|^3}dx\wedge dy \\
                		\alpha_1^{(2,1)} =& \frac{(x'-x)x+(z'-z)z}{|\vv{r}'-\vv{r}|^3}dy\wedge dz+\frac{ (y'-y)x}{|\vv{r}'-\vv{r}|^3}dz\wedge dx+\frac{(z'-z)x-(x'-x)z}{|\vv{r}'-\vv{r}|^3}dx\wedge dy\\
                		\alpha_2^{(2,1)} =& \frac{(y'-y)x}{|\vv{r}'-\vv{r}|^3}dy\wedge dz+\frac{ (z'-z)z-(x'-x)x}{|\vv{r}'-\vv{r}|^3}dz\wedge dx+\frac{-(y'-y)z}{|\vv{r}'-\vv{r}|^3}dx\wedge dy\\
                		\alpha_3^{(2,1)} =& \frac{-(x'-x)z+(z'-z)x}{|\vv{r}'-\vv{r}|^3}dy\wedge dz +\frac{-(y'-y)z}{|\vv{r}'-\vv{r}|^3}dz\wedge dx+\frac{-(z'-z)z-(x'-x)x}{|\vv{r}'-\vv{r}|^3}dx\wedge dy
        			\end{cases}
        		\end{equation}
        		\begin{equation}	
        			\begin{cases}
                		\alpha_0^{(2,2)} = &\frac{(y'-y)z+(z'-z)y}{|\vv{r}'-\vv{r}|^3}dy\wedge dz+\frac{-(x'-x)z}{|\vv{r}'-\vv{r}|^3}dz\wedge dx+\frac{-(x'-x)y}{|\vv{r}'-\vv{r}|^3}dx\wedge dy \\
                		\alpha_1^{(2,2)} =& \frac{(z'-z)z-(y'-y)y}{|\vv{r}'-\vv{r}|^3}dy\wedge dz+\frac{(x'-x)y}{|\vv{r}'-\vv{r}|^3}dz\wedge dx+\frac{-(x'-x)z }{|\vv{r}'-\vv{r}|^3}dx\wedge dy \\
                		\alpha_2^{(2,2)} =& \frac{(x'-x)y}{|\vv{r}'-\vv{r}|^3}dy\wedge dz + \frac{(y'-y)y+(z'-z)z}{|\vv{r}'-\vv{r}|^3}dz\wedge dx + \frac{(z'-z)y-(y'-y)z}{|\vv{r}'-\vv{r}|^3}dx\wedge dy \\
                		\alpha_3^{(2,2)} =& \frac{-(x'-x)z}{|\vv{r}'-\vv{r}|^3}dy\wedge dz+\frac{-(y'-y)z+(z'-z)y}{|\vv{r}'-\vv{r}|^3}dz\wedge dx+\frac{-(z'-z)z-(y'-y)y}{|\vv{r}'-\vv{r}|^3}dx\wedge dy
        			\end{cases}
     			\end{equation}
     			where superscript of $\alpha$ denotes which basis the differential $2$-form corresponds to, and subscript corresponds to index of its quaternion pair form. Corresponding $1$-forms are then given by,
     			\begin{equation}
        			\begin{cases}
                		\omega_0^{(2,1)} = & \left((xy^2+2xz^2)M_3-(y'xy+z'xz+x'z^2)M_2\right)dx + \left((yz^2-x^2y)M_3+(y'x^2-y'z^2)M_2\right)dy \\
                                		 & + \left((-2x^2z-y^2z)M_3+(z'x^2+x'xz+y'yz)\right) dz \\
                		\omega_1^{(2,1)} = & \left(-xyzM_3-(z'xy-y'xz-x'yz)M_2\right)dx + \left((x^2z+z^3)M_3+(z'x^2-2x'xz-z'z^2)M_2\right)dy\\
                                		 & + \left(-yz^2M_3+(-y'x^2+x'xy+z'yz)M_2\right) dz \\
                		\omega_2^{(2,1)} = & \left((x^2z-y^2z-z^3)M_3+(-x'xz+y'yz+z'z^2)M_2\right)dx + \left(2xyzM_3-2y'xzM_2\right)dy \\
                                		 & + \left((-x^3-xy^2+xz^2)M_3+(x'x^2+y'xy-z'xz)\right) dz \\
                		\omega_3^{(2,1)} = & \left(-x^2yM_3+(x'xy+cyz-y'z^2)M_2\right)dx + \left((x^3+xz^2)M_3-(x'x^2+2z'xz-x'z^2)M_2\right)dy \\
                                		 & + \left(-xyzM_3+(z'xy+y'xz-x'yz)M_2\right) dz 
        			\end{cases}
        		\end{equation}
        		\begin{equation}
        			\begin{cases}
                		\omega_0^{(2,2)} = & \left((xz^2-xy^2)M_3+(x'y^2-x'z^2)M_2\right)dx + \left((-x^2y-2yz^2)M_3+(x'xy+z'yz+y'z^2)M_2\right)dy \\
                                		 & + \left((x^2z+2y^2z)M_3+(-z'y^2-x'xz-y'yz)\right) dz \\
                		\omega_1^{(2,2)} = & \left(-2xyzM_3+2x'yzM_2\right)dx + \left((x^2-y^2z+z^3)M_3+(-x'xz+y'yz-z'z^2)M_2\right)dy\\
                                		 & + \left((x^2y+y^3-yz^2)M_3+(-x'xy-y'y^2+z'yz)M_2\right) dz \\
                		\omega_2^{(2,2)} = & \left((-y^2z-z^3)M_3-(z'y^2-2y'yz-z'z^2)M_2\right)dx + \left(xyzM_3+(z'xy-y'xz-x'yz)M_2\right)dy \\
                                		 & + \left(xz^2M_3+(-y'xy+x'y^2-z'xz)\right) dz \\
                		\omega_3^{(2,2)} = & \left((-y^3-yz^2)M_3+(y'y^2+2z'yz-y'z^2)M_2\right)dx + \left(xy^2M_3-(y'xy+z'xz-x'z^2)M_2\right)dy \\
                                		 & + \left(xyzM_3+(-z'xy+y'xz-x'yz)M_2\right) dz 
        			\end{cases}
        		\end{equation} 		
        		Now we have expression for the complete $1$-form $\omega$,
        		\begin{equation}
        			\begin{aligned}
        				\omega = \mu\left(\vv{r}^{(1,1)}\right)\omega_0^{(1,1)} + \Omega^{(2,1)}C^{(2,1)} + \Omega^{(2,2)}C^{(2,2)}, 
        			\end{aligned}
        		\end{equation}
        		where $\Omega^{(k,l)}=[\omega_0^{(k,l)},\omega_1^{(k,l)},\omega_2^{(k,l)},\omega_3^{(k,l)}]$. Lastly, the contour integral $\int_{\partial \tilde{D}} \omega$ is evaluated on each transformed patch.

\bibliographystyle{plain}
\bibliography{main}

\end{document}